\documentclass[11pt,reqno]{amsart}
\usepackage{amsmath,graphicx,color}
\usepackage{ifthen}
\usepackage{dsfont}
\usepackage{amssymb,latexsym}
\usepackage{subfigure}
\usepackage{hyperref}
\hypersetup{
urlcolor=black, 
  menucolor=black, 
  citecolor=black, 
  anchorcolor=black, 
  filecolor=black, 
  linkcolor=black, 
  colorlinks=true,
}
\usepackage{comment}
\usepackage[numbers,sort&compress]{natbib} 
\numberwithin{equation}{section}

\newcommand{\wt }{\widetilde}

\newcommand{\bfv}{{\mathbf v}}
\newcommand{\bfw}{{\mathbf w}}
\newcommand{\bfA}{{\mathbf A}}
\newcommand{\bfU}{{\mathbf U}}
\newcommand{\bfR}{{\mathbf R}}
\newcommand{\bfN}{{\mathbf N}}
\newcommand{\bfM}{{\mathbf M}}
\newcommand{\bfLambda}{{\boldsymbol \Lambda}}
\newcommand{\bfP}{{\mathbf P}}
\newcommand{\bfQ}{{\mathbf Q}}
\newcommand{\bfO}{{\mathbf O}}
\newcommand{\bfI}{{\mathbf I}}
\newcommand{\bfX}{{\mathbf X}}
\newcommand{\bfY}{{\mathbf Y}}
\newcommand{\bfx}{{\mathbf x}}
\newcommand{\bfu}{{\mathbf u}}
\newcommand{\bfy}{{\mathbf y}}
\newcommand{\bfB}{{\mathbf B}}
\newcommand{\bfe}{{\mathbf e}}
\newcommand{\bfb}{{\mathbf b}}
\newcommand{\bfz}{{\mathbf z}}
\newcommand{\X}{{\mathbf X}}
\newcommand{\A}{{\mathbf A}}

\renewcommand{\P}{{\mathbb P}}

\newcommand{\E}{\mathbb{E}}
\newcommand{\1}{\mathds{1}}
\newcommand{\R}{\mathbb{R}}
\newcommand{\N}{\mathbb{N}}

\newcommand{\Var}{\operatorname{Var}}
\newcommand{\Cov}{\operatorname{Cov}}

\newcommand{\norm}[1]{\|#1\|}
\newcommand{\vep}{\varepsilon}
\newcommand{\nto}{n \to \infty}
\newcommand{\tr}{\operatorname{tr}}

\newcommand{\diag}{\operatorname{diag}}
\newcommand{\lin}{\operatorname{lin}}

\newtheorem{lemma}{Lemma}[section]
\newtheorem{theorem}[lemma]{Theorem}
\newtheorem{proposition}[lemma]{Proposition}

\newtheorem{corollary}[lemma]{Corollary}

\newtheorem{remark}[lemma]{Remark}

\newcommand{\cid}{\stackrel{d}{\longrightarrow}}
\newcommand{\cip}{\stackrel{\P}{\rightarrow}}
\newcommand{\eid}{\stackrel{d}{=}}
\newcommand{\Vol}{\operatorname{Vol}}
\newcommand{\rank}{\operatorname{rank}}
\begin{document}
\date{}
\bibliographystyle{acm}
\title[Logarithmic volume of random simplices]
{The volume of random simplices from\\ elliptical distributions in high dimension}
\thanks{We would like to thank two anonymous referees for their insightful comments and remarks,
which helped us to improve our paper. AG  was supported by the DFG under Germany's Excellence Strategy  EXC 2044 -- 390685587, \textit{Mathematics M\"unster: Dynamics - Geometry - Structure}. CT has been supported by the DFG via SPP 2265 \textit{Random Geometric Systems}. We thank Nina Dörnemann and Nestor Parolya for fruitful discussions.}
\author[A. Gusakova]{Anna Gusakova}
\address{Institut für Mathematische Stochastik,
Westfälische Wilhelms-Universität Münster,
Orl\'eans-Ring 10,
48149 Münster,
Germany}
\email{gusakova@uni-muenster.de}
\author[J. Heiny]{Johannes Heiny}
\address{Department of Mathematics,
Stockholm University,
Albano hus 1,
10691 Stockholm,
Sweden}
\email{johannes.heiny@math.su.se}
\author[C. Thäle]{Christoph Thäle}
\address{Fakult\"at f\"ur Mathematik,
Ruhr-Universit\"at Bochum,
Universit\"atsstrasse 150,
D-44801 Bochum,
Germany}
\email{christoph.thaele@rub.de}
\begin{abstract}
Random simplices and more general random convex bodies of dimension $p$ in $\mathbb{R}^n$ with $p\leq n$ are considered, which are generated by random vectors having an elliptical distribution. In the high-dimensional regime, that is, if $p\to\infty$ and $n\to\infty$ in such a way that $p/n\to\gamma\in(0,1)$, a central and a stable limit theorem for the logarithmic volume of  random simplices and random convex bodies is shown. The result follows from a related central limit theorem for the log-determinant of $p\times n$ random matrices whose rows are copies of a random vector with an elliptical distribution, which is established as well.
\end{abstract}

\maketitle

\section{Introduction}\setcounter{equation}{0}

The probabilistic analysis of convex hulls in $\R^n$ generated by $p$ random points is one of the central themes of geometric probability and stochastic geometry. A variety of different models and results are known in the literature, mainly in the asymptotic regime $p\to\infty$, while the dimension parameter $n$ is kept fixed. Examples include the expectation asymptotics for the number of faces or the intrinsic volumes and their tight relation to affine surface areas, related upper and lower variance bounds as well as results on the asymptotic normality or concentration properties for these combinatorial and geometric parameters. We refer the reader to the survey articles of B\'ar\'any \cite{Bar07}, Hug \cite{Hug13} and Reitzner \cite{ReitznerSurvey} for motivation, background material and references. Recent interest in random polytopes has triggered an analysis in which the number $p$ of generating points and the space dimension $n$ tend to infinity in a suitably coupled way. In this context, the case where $p\leq n$ and $n\to\infty$ is of particular interest and has been considered in \cite{ABGK21,EK17,GKT17,Paindaveine}, building on earlier works \cite{Mae80,Math82,Rub77} in which $p$ stays constant. In this situation the random polytopes are just simplices of dimension $p$. In particular, these papers prove asymptotic normality of the logarithmic volume of the random simplices in high dimensions under special distributional assumptions on the generating random points. More precisely, the papers \cite{EK17,GKT17} deal with the case in which the random points are distributed according to a beta or beta-prime distribution in $\R^n$, which includes the uniform distribution in the $n$-dimensional unit ball and, as a limiting case, the uniform distribution on the $(n-1)$-dimensional unit sphere or the standard Gaussian distribution. The article \cite{Paindaveine} treats the Gaussian case together with applications to multivariate medians in statistics. In \cite{ABGK21} the distribution of the points arises from general product measures or a class of $q$-radial distributions on the $\ell_q$-ball in $\R^n$. In the latter setting, the particularly interesting case of high-dimensional pinned random simplices has been considered, where one of the generating points is fixed (pinned) at the origin. In this case, $p!$ times the volume of the random simplex is the same as the volume of the parallelotope spanned by the random points. 

The aim of the present paper is to provide central limit theorems for the logarithmic volume of $p$-dimensional pinned random simplices whose generating points follow a general elliptical distribution in $\R^n$. In particular, we will be interested in the high-dimensional regime, where $p=p_n$ is a function of the space dimension $n$ satisfying $p\to\infty$, as $n\to\infty$, and is such that $p/n\to\gamma\in(0,1)$. For convenience, let us recall that an $n$-dimensional random vector $\bfx$ follows an elliptical distribution if it takes the form
$$
\bfx=R\A\bfu,
$$
where $R\geq 0$ is an arbitrary real-valued random variable, $\A$ is a fixed $n\times n$ matrix of full rank and $\bfu$ is a uniform random direction, that is, a uniform random point on the unit sphere in $\R^n$, which is independent from $R$. Suppose now that we have $p$ independent copies $\bfx_1,\ldots,\bfx_p$  of the vector $\bfx$ that form the matrix  $\X:=(\bfx_1,\ldots,\bfx_p)^{\top}\in\R^{p\times n}$. 
The (pinned) random simplex with vertex set $\{\mathbf{0},\bfx_1,\ldots,\bfx_p \}$ can now be defined as
\begin{align}\label{def:simplex}
\Delta \X := \bigg\{ \sum_{i=1}^ps_i \bfx_i \,:\, s_i \geq 0 \quad\text{and}\quad \sum_{i = 1}^p s_i \leq 1 \bigg\},
\end{align}
and its $p$-volume admits the following representation in terms of the matrix~$\X$, see \cite[Section 8.7]{shilov:2012}:
\begin{align} \label{eq:simplexVol}
\Vol_p \left( \Delta\X  \right) 
= \frac{1}{p!} \sqrt{\det(\X\X^{\top})} \,,
\end{align}
where $\det(\bfM)$ denotes the determinant of a square matrix $\bfM$ and $\bfM^{\top}$ its transpose. Note that since $p\le n$  the vectors $\bfx_1,\ldots,\bfx_p$ are almost surely linearly independent and $\Delta \X$ is a $p$-dimensional random convex polytope in  $\R^n$ with non-zero $p$-volume. The representation \eqref{eq:simplexVol} immediately motivates a probabilistic analysis of random determinants of the form $\det(\X\X^\top)$, which are in the focus of the present paper as well. We remark that the study of random determinants has a long history starting with the works in \cite{GirkoCLT79,GirkoBook}, which have later been extended by many authors, see \cite{bao:lin:pan:zhou:2015,NguyenVuDet,TaoVuCLTWigner,wang:han:pan:2018, dette:doernemann:2020, heiny:johnston:prochno:2022, doernemann:2023, heiny:parolya:2023}, for example, as well as the references cited therein. Our main results can be summarized as follows:
\begin{itemize}
\item[(i)] Assuming upper and lower bounds on the eigenvalues of the matrix $\A\A^\top$ we derive in Theorem \ref{thm:main} asymptotic normality in high dimensions, that is, as $p\to\infty$ and $n\to\infty$ such that $p/n\to\gamma\in(0,1)$, for the logarithmic $p$-volume of the pinned random simplices $\Delta\X$ generated by random vectors following an elliptical distribution with $R=1$.  An example of such matrix $\A$  is a diagonal matrix $\diag\{a_{1},\ldots, a_n\}$, where the numbers $(a_i)_{i\ge 1}$ form a sequence with $|a_i|\in [C^{-1},C]$ for some $1\leq C< \infty$. In particular with $C=1$ we have that $\A$ is an identity matrix.  We also extend Theorem \ref{thm:main} by allowing arbitrary distributions for the radius $R\ge 0$ which might yield non-normal limiting laws for the logarithmic volume if the distribution of $\log R$ has infinite variance. 

\item[(ii)]
Then we take a more general point of view by considering for a fixed convex body $\Sigma\subset\R^p$ the random convex set
$$
\Upsilon_{p,n}(\Sigma,\X) := \left\{\sum_{i=1}^ps_i \bfx_i:(s_1,\ldots,s_p)\in\Sigma \right\},
$$
where $\bfx_1,\ldots,\bfx_p$ are the column vectors of $\X^\top$. In the special case $p=n$ this random set model has first been considered in \cite{PP13} in the context of small-ball probabilities and later in \cite{ABGK21} under the angle of high-dimensional central limit theorems. Note that by choosing for $\Sigma$ the $p$-dimensional standard simplex, $\Upsilon_{p,n}(\Sigma,\X)$ reduces to the pinned random simplex $\Delta\X$ discussed above. We establish in Theorem \ref{thm:RandomConvexBody} a central limit theorem for the logarithmic $p$-volume of $\Upsilon_{p,n}(\Sigma,\X)$, as $p\to\infty$ and $n\to\infty$ such that $p/n\to\gamma\in(0,1)$, for matrices $\X$ generated by random vectors satisfying the conditions of Theorem \ref{thm:main}.
\end{itemize}

The remaining parts of this text are structured as follows. In Section~\ref{sec:Results}, we formally introduce our set-up together with the necessary notation. Our main results, Theorem~\ref{thm:main} for the log-determinant and its geometric counterpart Theorem \ref{thm:RandomConvexBody}, are the contents of Section~\ref{subsec:LogDet} and Section~\ref{subsec:RCB}, respectively. Section \ref{sec:proofmain} is devoted to the proof of Theorem \ref{thm:main}, while the Appendix contains some auxiliary lemmas.

\section{Limit theorems for the logarithmic volume of elliptical simplices}\label{sec:Results}\setcounter{equation}{0}

We consider independent and identically distributed (i.i.d.)~$n$-dimensional random vectors $\bfx_1,\ldots,\bfx_p$  following an elliptical distribution, that is,
\begin{equation}\label{eq:dataell}
\bfx_i=R_i\A\bfu_i\,, \qquad i=1,\ldots,p\,,
\end{equation}
where $\A\in \R^{n\times n}$ is a deterministic matrix of full rank, 
$R_i\ge 0$ is a scalar random variable representing the radius of $\bfx_i$, and $\bfu_i$ is the random direction, which is uniformly distributed on the unit sphere  $\mathbb{S}^{n-1}$. Moreover, we suppose that the random variables $R_i, \bfu_i, i=1,\ldots,p$ are independent. 
The elliptically distributed random vectors are collected in the $p\times n$ matrix 
\begin{equation}\label{eq:datam}
\X:=\bfx_n:=(\bfx_1,\ldots,\bfx_p)^{\top}\,.
\end{equation}
{\em Throughout this paper, $\X, \bfx_i, R_i, \bfu_i, \A$ depend on the dimension $n$, i.e.\ $\X=\bfx_n, \bfx_i=\bfx_{i,n}, R_i=R_{i,n}, \bfu_i=\bfu_{i,n}, \A=\A_n$. For simplicity we will typically suppress the dependence on $n$ in the notation.}

Let us consider the random simplex $\Delta\X$, defined in \eqref{def:simplex}, induced by the matrix $\X$. From \eqref{eq:simplexVol} we get the identity
\begin{equation}\label{eq:081220A}
\log \Vol_p \left( \Delta\X  \right)=-\log(p!) +\frac{1}{2} \log \det(\X\X^{\top})\,.
\end{equation}
Thus, in order to establish a central limit theorem for $\log \Vol_p \left( \Delta\X  \right)$ it is enough to establish such a result for $\log \det(\X\X^{\top})$. This problem is treated in Section \ref{subsec:LogDet}, whereas its geometric implications will be discussed in Section \ref{subsec:RCB}.

\subsection{Central limit theorem for the log-determinant}\label{subsec:LogDet}

In this section we deal with a central limit theorem for the log-determinant of $\X\X^\top$. Introducing the diagonal matrix $\bfR$ with diagonal entries $R_1,\ldots, R_p$, and the matrix
$$
\bfY:=\bfY_n:=(\A \bfu_1,\ldots,\A \bfu_p)^{\top},
$$
we have
\begin{equation*}
\det(\X\X^{\top})= \det(\bfR \bfY\bfY^{\top} \bfR)
=\det(\bfY\bfY^{\top}) (\det\bfR)^2\,
\end{equation*}
and taking the logarithm yields
\begin{equation}\label{eq:081220B}
\begin{split}
\log \det(\X\X^{\top})&= 2 \log \det \bfR + \log \det(\bfY\bfY^{\top})\\
&= 2 \sum_{i=1}^p \log R_i + \log \det(\bfY\bfY^{\top})\,,
\end{split}
\end{equation}
where the the last two terms are independent.

The following CLT for $\log \det(\bfY\bfY^{\top})$ is crucial for us. We start by spelling out our assumptions under which we are able to derive the central limit theorem:
\begin{itemize}
\item[(A)] For the parameter $p=p_n$ there exists a constant $\gamma\in(0,1)$ such that $p/n\to \gamma$ as $\nto$. 
\item[(B)] There exists a constant $C\ge 1$ not depending on $n$ such that the ordered eigenvalues 
\begin{equation}\label{eq:B1}
\lambda_{max}(\A\A^{\top})=\lambda_1(\A\A^{\top})\ge \lambda_2(\A\A^{\top})\ge\cdots \ge \lambda_n(\A\A^{\top})=\lambda_{min}(\A\A^{\top})
\end{equation}
of $\A\A^{\top} =\A_n \A_n^{\top}$ satisfy
$C^{-1}\le \lambda_{min}(\A\A^{\top})\le \lambda_{max}(\A\A^{\top})\le C$. In addition, we assume that 
\begin{equation}\label{eq:B2}
\lim_{\nto} \tr\Big( \A_n\A_n^{\top} - \frac{\tr(\A_n\A_n^{\top})}{n} \bfI_n\Big)^2= \lim_{\nto} \sum_{k=1}^n\Big( \lambda_{k}(\A_n\A_n^{\top})- \frac{1}{n} \sum_{j=1}^n\lambda_{j}(\A_n\A_n^{\top})\Big)^2 =0\,,
\end{equation}
where $\bfI_n$ is the $n\times n$ identity matrix.
\end{itemize}

Next, we introduce some quantities that will play an important role in our results. For $1\le i\le p-1$ and $1\le k\le n$, we set
\begin{equation}\label{def:tik}
t_{i,k}(\A)= \E\left[\frac{1}{1+ \lambda_k(\A\A^{\top}) \bfw_{ik}^{\top} \big(\sum_{\ell=1; \ell \neq k}^n \lambda_{\ell}(\A\A^{\top}) \bfw_{i\ell} \bfw_{i\ell}^{\top}\big)^{-1}\bfw_{ik}}\right]\,,
\end{equation}
where $\bfw_{i1}, \ldots, \bfw_{in}$ are i.i.d.\ $i$-dimensional random column vectors whose components are independent standard normal random variables. We shall elaborate on these conditions and quantities in Remark \ref{rem:3.2} below.
Recalling that $\bfY=\bfY_n$ and $\A=\A_n$, we are now prepared to present our central limit theorem for the log-determinant of $\bfY\bfY^{\top}$, whose proof is presented in Section \ref{sec:proofmain}. 

\begin{theorem}\label{thm:main}
Let $(\bfY)_{n \ge 1}$ be a sequence of random $p\times n$ matrices defined as follows: $\bfY=(\A \bfu_1,\ldots,\A \bfu_p)^{\top}$, where $(\A)_{n\ge 1}$ is a sequence of deterministic $n\times n$ matrices satisfying assumption (B), and $\bfu_1,\ldots, \bfu_p$ are independent $n$-dimensional random vectors, distributed uniformly on the unit sphere $\mathbb{S}^{n-1}$.
Under assumption (A), as $\nto$, it holds that
\begin{equation}\label{eq:main}
\frac{\log \det(\bfY\bfY^{\top}) -\mu_n}{\sigma_n} \cid N(0,1)\,,
\end{equation}
where the centering and normalizing sequences $(\mu_n)_{n\geq 1}$ and $(\sigma_n)_{n\geq 1}$ are given by
\begin{align}\label{eq:meanvar}
\begin{split}
\mu_n&= \log \tr(\A\A^{\top}) -p \log n -\frac{\sigma_n^2}{2}+ \sum_{i=1}^{p-1} \log \bigg(\sum_{k=1}^n \lambda_k(\A\A^{\top}) t_{i,k}(\A)\bigg)\,,\\
\sigma_n^2&=-2\frac{p}{n} + 2 \sum_{i=1}^{p-1} \frac{\sum_{k=1}^n \lambda_k^2(\A\A^{\top}) t_{i,k}(\A)}{(\sum_{k=1}^n \lambda_k (\A\A^{\top}) t_{i,k}(\A))^2}\,.
\end{split}
\end{align}
\end{theorem}

\begin{remark}\label{rem:3.2}\rm 
\begin{enumerate}
\item[(1)] The centering and normalizing sequences only depend on $\A$ through the eigenvalues $\lambda_1(\A\A^{\top}), \ldots, \lambda_n(\A\A^{\top})$. This is due to the rotational invariance of the vectors $\bfu_i$ that was used in \eqref{eq:repU} of the proof. 
\item[(2)] Assumption (A) and \eqref{eq:B1} in Assumption (B) guarantee that $\sigma_n^2$ is of constant order, see Lemma~\ref{lem:sigma}, which is convenient in view of the existing results on log-determinants of random matrices \cite{BaoPanZhou2015,GirkoCLT79,GirkoBook,NguyenVuDet,TaoVuCLTWigner}. Assumption (A) excludes $\gamma\in \{0,1\}$. If $p/n \to 1$, an adjustment is required for the variance $\sigma_n^2$. For example, in the case $p/n \to 1$ and $\bfA=\bfI_n$, $\sigma_n^2$ can be chosen as $-2\log(1-(p-1)/n)- 2 p/n$ which tends to infinity \cite{wang:han:pan:2018}. Moreover, in the case $p/n\to 0$ the limiting distribution for the log-determinant might not be Gaussian. For example, if $p$ stays constant, one obtains a Gamma distribution \cite{heiny:johnston:prochno:2022}.
\item[(3)] The quantities $t_{i,1}(\A), \ldots, t_{i,n}(\A)$ satisfy $t_{i,1}(\A)+ \cdots +t_{i,n}(\A)=n-i$ for any full rank matrix $\A\in \R^{n\times n}$ and $1\le i\le p-1$. For details we refer to Remark~\ref{rem:tik}.
\item[(4)] If $\bfA=\bfI_n$ (the $n$-dimensional identity matrix), one sees from \eqref{def:tik} that $t_{i,k}(\bfI_n)=t_{i,\ell}(\bfI_n)$ for any $k\neq \ell$, which in conjunction with the previous point implies $t_{i,k}(\bfI_n)=(n-i)/n$. Plugging this into \eqref{eq:meanvar}, we recover \cite[Theorem~1]{jiang:yang:2013} and a particular case of \cite[Theorem~2.1]{parolya:heiny:kurowicka:2021} as a special case of our Theorem~\ref{thm:main}.
\item[(5)] Instead of \eqref{eq:B1} in Assumption (B), it is possible to require 
$$\frac{\lambda_{max}(\A\A^{\top})}{\lambda_{min}(\A\A^{\top})} \le C^2$$
for some constant $C\ge 1$ not depending on $n$, which would, for example, allow that the largest and smallest eigenvalues of $\A\A^{\top}$ tend to zero or infinity at the same speed. Up to a rescaling of $\A$, this assumption is, however, equivalent to \eqref{eq:B1}. 
\item[(6)] Equation \eqref{eq:B2} in Assumption (B) is a technical condition that can likely be relaxed or removed.  It means that $\A\A^{\top}$ is close in Frobenius norm to a multiple of the identity matrix. The second limit in \eqref{eq:B2} shows that \eqref{eq:B2} is a condition on the eigenvalues of $\A\A^{\top}$. 
Assumption \eqref{eq:B1} is the main weakness of Theorem \ref{thm:main} as it is rather restrictive. On the positive side, \eqref{eq:B2} is only required in the proof of Lemma~\ref{lem:yaskov} to control the variance of a linear combination of diagonal entries of a large projection matrix. If the right-hand side of \eqref{eq:nastypart} can be improved to $o(1/n)$ in the proof of Lemma~\ref{lem:yaskov}, then assumption \eqref{eq:B2} can be dropped.
\end{enumerate}
\end{remark}

 In the proof of Theorem~\ref{thm:main} we first show that due to ellipticity it suffices to restrict ourselves to diagonal matrices $\bfA$ when studying $\log \det(\bfY\bfY^{\top})$. Using the elementary fact that the magnitude of the determinant of a Gram matrix built of real vectors is equal to the volume of the parallelepiped spanned by those vectors, we obtain a representation for the log-determinant of $\bfY \bfY^{\top}$ as a product of perpendiculars, see \cite[p.~1602 ff.]{bao:pan:zhou:2015} for details. The same trick has also been employed in \cite{bao:pan:zhou:2015, heiny:parolya:2023, wang:han:pan:2018} to obtain central limit theorems for large sample covariance and correlation matrices. This representation for $\log \det(\bfY\bfY^{\top})$ is then studied through martingale differences which are based on quadratic forms involving large projection matrices. The crucial difference to the aforementioned references is that the diagonal elements of our projection matrices are dependent  and not identically distributed (unless $\bfA=\bfI_n$), which significantly complicates the analysis and the derivation of the centering and normalizing sequences $(\mu_n)_{n\geq 1}$ and $(\sigma_n)_{n\geq 1}$. The detailed technical proof of Theorem~\ref{thm:main} is given in Section~\ref{sec:proofmain}.

\subsection{Central limit theorem for the volume of random simplices and convex bodies}\label{subsec:RCB}

After having discussed the central limit theorem for random determinants in the previous section, we return to our geometric application to random simplices and convex bodies. We start by recalling from \eqref{eq:081220A} and \eqref{eq:081220B} that 
$$
\log\Vol_p(\Delta\X) = -\log(p!) + \sum_{i=1}^p\log R_i + \frac{1}{2}\log\det(\bfY\bfY^\top),
$$
where we use the same notation as in the previous section. Especially we recall that $p=p_n$ and that $\X$ is a matrix generated by $p$ independent random elliptically distributed random vectors with radial parts $R_1,\ldots,R_p$. Since the random variables $R_1,\ldots,R_p$ are independent and identically distributed for every $n$, the generalized central limit theorem for sums of i.i.d.\ random variables (see, e.g., \cite[Theorem~IV.4.18]{petrov:1975} and \cite[Theorem~B.2]{heiny:johnston:prochno:2022})  implies that there are sequences $(m_n)_{n\geq 1}$ and $(s_n)_{n\geq 1}$ such that
\begin{equation}\label{eq:CLTRi}
\frac{\sum_{i=1}^p\log R_i- m_n}{s_n} \overset{d}{\longrightarrow} S_\alpha,\qquad n\to\infty\,,
\end{equation}
where $S_\alpha$ stands for the $\alpha$-stable distribution with stability index $\alpha\in(0,2]$. For example, if $\alpha=1$, $S_\alpha$ is the Cauchy distribution or if $\alpha=2$ then $S_\alpha$ corresponds to a centered Gaussian distribution. Sufficient conditions on $\log R_1$ such that \eqref{eq:CLTRi} holds are provided in \cite[Theorem~IV.4.18]{petrov:1975} for the case $\alpha=2$ and in \cite[Theorem~B.2]{heiny:johnston:prochno:2022} for the case $\alpha\in (0,2)$. Writing now
\begin{align*}
V_n:&={\log\Vol_p(\Delta\X)-{1\over 2}\mu_n-m_n+ \log(p!) \over\max\{\sigma_n/2,s_n\}}\\ 
&={\sum_{i=1}^p\log R_i- m_n\over\max\{\sigma_n/2,s_n\}} + {{1\over 2}\log\det(\bfY\bfY^\top)-{1\over 2}\mu_n\over\max\{\sigma_n/2,s_n\}}
\end{align*}
and noting that both summands on the right-hand side are independent, we conclude the following corollary from Theorem \ref{thm:main}, part (2) of Remark~\ref{rem:3.2} and \eqref{eq:CLTRi}. We note that this result generalizes Theorems E-F in \cite{heiny:johnston:prochno:2022} for $p/n\to \gamma$, where only the case $\A=\bfI_n$ was analyzed.

\begin{corollary}\label{cor:SimplexWithRadialPart}
Consider the random simplex $\Delta\X$ defined in \eqref{def:simplex}, where $\X=(\bfx_1,\ldots,\bfx_p)^{\top}$ is as in \eqref{eq:datam} with i.i.d.\ random vectors $\bfx_i=R_i \A \bfu_i$ following an $n$-dimensional elliptical distribution as in \eqref{eq:dataell}. Under assumptions (A) and (B), and assuming that $\log R_1$ satisfies \eqref{eq:CLTRi} for some $\alpha \in (0,2]$, the following statements hold. 
\begin{enumerate}
\item[(i)] If $s_n\to 0$, then $V_n\cid N(0,1)$, as $n\to\infty$.
\item[(ii)] If $s_n\to \infty$, then $V_n \cid S_\alpha$, as $n\to\infty$.
\item[(iii)] If $s_n/(\sigma_n/2)\to \tau\in(0,\infty)$, then $V_n \cid \max\big\{1,{1\over\tau}\big\}S_\alpha+\min\{1,\tau\}N(0,1)$, as $n\to\infty$.
\end{enumerate}
\end{corollary}

In a next step, we turn to a generalization of this result.  We begin by noting that for any $p\times p$ matrix $\bfM$, one has
\begin{equation*}
\det(\bfM \bfY\bfY^{\top} \bfM^{\top})= (\det(\bfM))^2 \det(\bfY\bfY^{\top})\,.
\end{equation*}
Assuming that $\bfM$ has full rank, this yields the following representation of the log-volume:
\begin{align} 
\log \Vol_p \left( \Delta (\bfM\bfY)  \right) 
&= -\log(p!)+ \frac{1}{2} \log \det(\bfM \bfY\bfY^{\top}\bfM^{\top}) \nonumber\\
&=-\log(p!)+ \frac{1}{2} \log \det(\bfY\bfY^{\top}) +\log |\det \bfM|\,. \label{eq:simplexnew}
\end{align}
Thus the fluctuations of $\log \Vol_p \left( \Delta (\bfM\bfY)  \right)$ can be directly derived from our theorems above.

Theorem \ref{thm:main} has an immediate geometrical interpretation, which we will present next. For a compact convex subset $\Sigma \subset \R^p$ with non-empty interior and random $n$-dimensional vectors $\bfy_1,\ldots, \bfy_p$ we define the random convex body 
\begin{equation}\label{}
\Upsilon_{p,n}(\Sigma,\bfY) :=\bigg\{ \sum_{i=1}^p s_i \bfy_i \,:\, (s_1,\dots,s_p)\in \Sigma \bigg\}\,,
\end{equation}
where $p \le n$ and the $n\times p$ matrix $\bfY^{\top}=(\bfy_1,\ldots,\bfy_p)$ is treated as a linear operator, applied to the body $\Sigma$. It should be noted here, that for any $n,p\in\mathbb{N}$, $p\le n$ and for any convex body $\Sigma$, $\Upsilon_{p,n}(\Sigma,\bfY)$ is a random $p$-dimensional closed convex set in $\mathbb{R}^n$ in the usual sense of stochastic geometry, see \cite[Chapter 2]{SW08}. The $p$-dimensional volume of $\Upsilon_{p,n}(\Sigma,\bfY_n)$ is hence an ordinary random variable. Depending on the choice of the convex body $\Sigma$ the construction above includes a number of common geometrical objects, which are of particular interest, see \cite{PP12,PP13, ABGK21}.

\begin{itemize}
\item [(a)] If $\Sigma$ is the standard simplex $T^p$,
$$
T^p:=\bigg\{(x_1,\ldots,x_p)\in\mathbb{R}^p\colon x_i\ge 0, \sum\limits_{i=1}^p x_i\leq 1\bigg\},
$$
then $\Upsilon_{p,n}(T^p,\bfY)$ coincides with the pinned simplex $\Delta \bfY$, which is the convex hull of points $\mathbf{0},\bfy_1,\ldots, \bfy_p$, see \eqref{def:simplex}.
\item [(b)] If $\Sigma$ is the unit cube $C^p=[0,1]^p$, then $\Upsilon_{p,n}(C^p,\bfY)$ is the parallelotope spanned by vectors $\bfy_1,\ldots, \bfy_p$.
\item [(c)] If $\Sigma$ is the symmetric cube $B_{\infty}^p=[-1,1]^p$, then $\Upsilon_{p,n}(B^p_{\infty},\bfY)$ is a zonotope, generated by the random segments $[-\bfy_i,\bfy_i]$, $1\leq i\leq p$.
\item [(d)] If $\Sigma$ is the cross-polytope $B^p_1$,
$$
B_1^p:=\{(x_1,\ldots,x_p)\in\mathbb{R}^p\colon \sum_{i=1}^p|x_i|\leq 1\},
$$
then $\Upsilon_{p,n}(B^p_1,\bfY)$ is the convex hull of the $2p$ random points $\pm \bfy_1,\ldots,\pm\bfy_p$.
\item[(e)]If $\Sigma$ is the unit ball $B^p_2$, then $\Upsilon_{p,n}(B^p_2,\bfY)$ is an ellipsoid in $\R^n$, which is included into the $p$-dimensional linear subspace, spanned by the vectors $\bfy_1,\ldots, \bfy_p$. The lengths and directions of the semi-axes of the ellipsoid $\Upsilon_{p,n}(B^p_2,\bfY)$ are defined by the square root of the eigenvalues and the corresponding eigenvectors of the matrix $\bfY^{\top}\bfY$.
\end{itemize}

The following theorem provides asymptotic normality for the logarithmic volume of the random convex bodies $\Upsilon_{p,n}(\Sigma,\bfY)$, when $p$ and $n$ tend to infinity simultaneously and the generating vectors have an elliptic distribution with all radial parts equal to $1$. We remark that a limit theorem for the logarithmic volume $\log\Vol_p(\Upsilon_{p,n}(\Sigma,\X))$ with $\X$ having a general elliptic distribution with random radial parts $R_1,\ldots,R_p$ can be derived as above and leads to a result very similar to that of Corollary~\ref{cor:SimplexWithRadialPart}. For simplicity, we decided to restrict our attention to random convex bodies generated by $\bfY$ only.

\begin{theorem}\label{thm:RandomConvexBody}
Assume the conditions of Theorem \ref{thm:main} and let $\bfy_1,\ldots, \bfy_p$ be the columns of the matrix $\bfY^{\top}$ defined in Theorem \ref{thm:main}. Let $(\Sigma_p)_{p\in\mathbb{N}}$ be a sequence of convex bodies such that $\Sigma_p\subset\mathbb{R}^p$. Then, as $n\to\infty$, we have
$$
{\log\Vol_p(\Upsilon_{p,n}(\Sigma_p,\bfY))-{1\over 2}\mu_n-\log\Vol_p(\Sigma_p) 
\over {1\over 2}\sigma_n}\cid N(0,1),
$$
where $\mu_n$ and $\sigma_n$ are defined in \eqref{eq:meanvar}.
\end{theorem} 

\begin{proof}
In the first step, we will show the equality 
\begin{equation}\label{eq_19.11.20}
\Vol_p(\Upsilon_{p,n}(\Sigma_p, \bfY))=\sqrt{\det(\bfY\bfY^{\top})}\Vol_p(\Sigma_p).
\end{equation}
We note that for any integer $m$, any $m$-dimensional convex body $\tilde\Sigma$ and all $m$-dimensional vectors $\bfx_1,\ldots,\bfx_m$ with $\bfX_{(m)}=(\bfx_1,\ldots,\bfx_m)^{\top}$ we have
\begin{equation}\label{eq_19.11.20_2}
\Vol_m(\Upsilon_{m,m}(\tilde\Sigma,\bfX_{(m)}))=|\det( \bfX_{(m)})|\Vol_m(\tilde\Sigma)=m!\Vol_m(\Delta \bfX_{(m)})\Vol_m(\tilde\Sigma),
\end{equation}
see, for example, \cite[Proposition 2.1]{PP13}.
Denote by $L:=\lin(\bfy_1,\ldots,\bfy_p)$ the linear hull of the vectors $\bfy_1,\ldots,\bfy_p$. Then for any $s_1,\ldots,s_p\in\mathbb{R}$ we have
$$
\sum_{i=1}^ps_i\bfy_i\in L,
$$
and, hence, $\Upsilon_{p,n}(\Sigma_p,\bfY)\subset L$. If $\rank(\bfY)<p$, then $\det(\bfY\bfY^{\top})=0$ and $\dim\Upsilon_{p,n}(\Sigma_p,\bfY)\leq\dim L<p$, leading to $\Vol_p(\Upsilon_{p,n}(\Sigma_p,\bfY))=0$. Thus, \eqref{eq_19.11.20} trivially holds. 

Let $\dim L=p$, $\bfe_1,\ldots,\bfe_n$ be the standard orthonormal basis of $\mathbb{R}^n$ and let $O_L:\mathbb{R}^n\mapsto \mathbb{R}^n$ be a rotation operator, such that  $O_L(L)=\lin(\bfe_1,\ldots,\bfe_p)$. The linear hull $\lin(\bfe_1,\ldots,\bfe_p)$ can be identified with $\mathbb{R}^p$. Writing $\widetilde{\bfy}_i:=O_L\bfy_i\in\lin(\bfe_1,\ldots,\bfe_p) \cong \mathbb{R}^p$, using that the volume is invariant with respect to rotations and applying \eqref{eq_19.11.20_2} we conclude
\begin{align*}
\Vol_p(\Upsilon_{p,n}(\Sigma_p, \bfY))&=\Vol_p(\Upsilon_{p,n}(\Sigma_p, O_L(\bfY))\\
&=\Vol_p(\Upsilon_{p,p}(\Sigma_p, \widetilde{\bfY}))\\
&=p!\Vol_p(\Delta\widetilde{\bfY})\Vol_p(\Sigma_p)\\
&=p!\Vol_p(\Delta\bfY)\Vol_p(\Sigma_p),
\end{align*}
which is equivalent to \eqref{eq_19.11.20}. The desired conclusion now follows directly from \eqref{eq_19.11.20} and Theorem~\ref{thm:main}.
\end{proof}

\section{Proof of Theorem \ref{thm:main}}\setcounter{equation}{0}\label{sec:proofmain}

\subsection{Some notation} 
In preparation for the proof of Theorem~\ref{thm:main}, we need to introduce some useful notation.
We start by noting that the entries of the vector $\bfu_1=(U_{11},\ldots,U_{1n})^{\top}$ are symmetric and exchangeable due to the
fact that the Euclidean unit ball is a $1$-symmetric convex body. We have the stochastic representation
 \begin{equation}\label{def:R}
 (U_{11},\ldots,U_{1n}) \overset{d}{=} \frac{1}{\sqrt{N_{11}^2+\cdots+N_{1n}^2}} (N_{11},\ldots,N_{1n})\,,
  \end{equation}
where $(N_{ij})_{i,j\ge 1}$ is a field of independent standard normal random variables.
  By \cite[Example 2.1]{heiny:mikosch:2017:corr}, we have for positive integers $m_1,\ldots,m_r$
\begin{equation}\label{lem:moment24}
 \E[U_{11}^{2m_1} U_{12}^{2m_2}\cdots U_{1r}^{2m_r}]  = \frac{\Gamma(n/2) \prod_{j=1}^r (2m_j-1)!!}{2^{m_1+\cdots+m_r} \Gamma(n/2+m_1+\cdots+m_r)} =\frac{\prod_{j=1}^r (2m_j-1)!!}{\prod_{j=0}^{m_1+\cdots+ m_r-1} (n+2j)}\,.
\end{equation}
In particular, we have $\E[U_{11}^4] = 3/(n(n+2))$ and we note that $\E[U_{11}^{2m_1} U_{12}^{2m_2}\cdots U_{1r}^{2m_r}]$ is of order $n^{-(m_1+\cdots+m_r)}$. By symmetry of the normal distribution, $\E[U_{11}^{m_1} U_{12}^{m_2}\cdots U_{1r}^{m_r}]$ is zero if one of the $m_i$'s is odd.
Throughout this section, we will use the notation 
$$\beta_{2m_1,\ldots, {2m_r}}:=\E [ U_{11}^{2m_1} \cdots U_{1r}^{2m_r} ]\,.$$
Since $\beta_{2m_1,\ldots, {2m_r}}=\beta_{2m_{\pi(1)},\ldots, 2m_{\pi(r)}}$ for any permutation $\pi$ on $\{1,\ldots,r\}$ we will typically write the indices in decreasing order. For example, instead of $\beta_{2,4}$ we prefer writing $\beta_{4,2}$. 
\par

For any symmetric matrix $\bfM\in \R^{d\times d}$, we will write
$$\lambda_{max}(\bfM)=\lambda_{1}(\bfM)\ge \lambda_{2}(\bfM) \ge \cdots \ge \lambda_{d}(\bfM)=\lambda_{min}(\bfM)$$
for its ordered eigenvalues, $\norm{\bfM}$ for its spectral norm, that is,  $\norm{\bfM}=\sqrt{\lambda_1(\bfM \bfM^{\top})}$,
and denote its spectral decomposition by
\begin{equation}\label{eq:SpectralDecomposition}
\bfM=\mathcal{\bfO}_{\bfM} \bfLambda_{\bfM} \mathcal{\bfO}_{\bfM}^{\top}\,,
\end{equation}
where $ {\bfLambda}_{\bfM}$ is the diagonal matrix whose $i$-th diagonal element is $\lambda_i(\bfM)$, and $\mathcal{\bfO}_{\bfM}$ is an orthogonal matrix. Finally given two sequences $(a_n)_{n\in\N}$ and $(b_n)_{n\in\N}$ we write $a_n=O(b_n)$ if $\limsup\limits_{n\to\infty}|a_n/b_n|<\infty$ and $a_n=o(b_n)$ if $\lim\limits_{n\to\infty}|a_n/b_n|=0$.

\subsection{Opening: Proof of Theorem~\ref{thm:main}}
We set $\bfU=(\bfu_1,\ldots,\bfu_p)^{\top}$ and $\bfU_{(i)}=(\bfu_1,\ldots,\bfu_i)^{\top}$ for $1\le i \le p$. Due to the rotational invariance of the rows of $\bfU$, we have $\bfU\eid  \bfU\mathcal{\bfO}$ for any orthogonal matrix $\mathcal{\bfO}$ with the appropriate dimension.

In view of \eqref{def:R}, we will use the representation
\begin{equation}\label{eq:repU}
\bfU \overset{d}{=} (\diag(\bfN \bfN^{\top}))^{-1/2} \bfN\,,
\end{equation}
where $\bfN$ is a $p\times n$ matrix with independent standard normal entries $(N_{ij})$ and $\diag(\bfN \bfN^{\top})$ denotes the diagonal matrix with the same diagonal elements as $\bfN \bfN^{\top}$. Recalling that $\bfY=\bfU \A^{\top}$ and using the spectral decomposition $\A^{\top} \A=\mathcal{\bfO}_{\A^{\top} \A} {\bfLambda}_{\A^{\top} \A} \mathcal{\bfO}_{\A^{\top} \A}^{\top}\,,$ and the rotational invariance of the rows of $\bfU$, we see that
\begin{equation}\label{eq:dkdkdd}
\bfY \bfY^{\top}=\bfU \A^{\top} \A \bfU^{\top} \eid \bfU {\bfLambda}_{\A^{\top} \A} \bfU^{\top}\,.
\end{equation}
Therefore, we may assume without loss of generality that $\A={\bfLambda}_{\A^{\top} \A}^{1/2}$ and, hence, that $\bfA$ is a diagonal matrix.

Next, we see that 
\begin{equation}\label{eq:dkdkdd1}
\log \det \left(\bfU \A^{\top} \A \bfU^{\top}\right) =p \log \frac{\tr(\A^{\top} \A)}{n} +\log \det \left( \frac{n}{\tr(\A^{\top} \A)} \bfU \A^{\top} \A \bfU^{\top} \right)\,.
\end{equation}
{Thus, in the remainder of this section, we will assume that $\A$ is a diagonal matrix with positive diagonal elements $A_{jj}=\sqrt{\lambda_j(\A^{\top} \A)}$, $j=1,\ldots,n$ and $\tr(\A^{\top} \A)=n$. By Assumption (B), there exists a constant $C\ge 1$ not depending on $n$ such that $C^{-1}\le A_{nn}^2 \le A_{11}^2\le C$.}
\medskip

 We use Girko's method of perpendiculars, whose starting point is the elementary fact that the magnitude of the determinant of a Gram matrix built of real vectors is equal to the volume of the parallelepiped spanned by those vectors. This is the reason why by a simple ``base times height'' formula one can represent the determinant as a product of perpendiculars, see \cite[p.~1602 ff.]{bao:pan:zhou:2015} for details.
Therefore, using the method of perpendiculars as in \cite[p.~85-86]{wang:han:pan:2018} who derived an expression for the log determinant of the sample covariance matrix, we get
\begin{equation}\label{eq:fedfse}
\log \det (\bfY\bfY^{\top})= -p \log n+ \sum_{i=0}^{p-1} \log( Z_{i+1})\,,
\end{equation}
where
\begin{equation*}
 Z_{i+1}=n\, \bfb_{i+1}^{\top}  \bfP_i \bfb_{i+1} \quad \text{ and } \quad \bfP_i=\bfI_n-\bfB_{(i)}^{\top} (\bfB_{(i)} \bfB_{(i)}^\top)^{-1} \bfB_{(i)}\,.
\end{equation*}
Here $\bfP_0=\bfI_n$, $\bfB_{(i)}=(\bfb_1,\ldots, \bfb_i)^{\top}$ and $\bfb_i=(Y_{i1}, \ldots,Y_{in})^{\top}$ denotes the $i$-th row of the matrix $\bfY$, that is $\bfb_i=\A \bfu_i$, and $\bfP_i=(p_{i,kl})$ is a projection matrix. It should be noted that all $\bfB_{(i)} \bfB_{(i)}^\top$ are a non-singular matrices as will be explained below \eqref{eq:repPi}. It is also easy to check that $\bfP_i=\bfP_i^2$ and that with probability one $\tr(\bfP_i)=n-i$. 
Using \eqref{eq:repU}, we can write $\bfP_i$ as
\begin{equation}\label{eq:repPi}
\bfP_i=\bfI_n-\A \bfN_{(i)}^{\top} \big(\bfN_{(i)}\bfA^2 \bfN_{(i)}^{\top}\big)^{-1} \bfN_{(i)} \bfA\,,
\end{equation}
where $\bfN_{(i)}$ is the $i\times n$ matrix comprised of the first $i$ rows of $\bfN$.

It holds $\lambda_{min}(\bfN_{(i)}\bfA^2 \bfN_{(i)}^{\top}) \ge \lambda_{min}(\bfA^2) \lambda_{min}(\bfN_{(i)} \bfN_{(i)}^{\top})\ge C^{-1} \lambda_{min}(\bfN_{(i)} \bfN_{(i)}^{\top})$.
Moreover, due to \cite[Proposition 2.1]{wang:han:pan:2018} all matrices $\bfN_{(i)} \bfN_{(i)}^{\top}$ are invertible with overwhelming probability, so that in combination with the previous statement all $\bfN_{(i)}\bfA^2 \bfN_{(i)}^{\top}$ are invertible.
By \cite[Lemma 2.1]{mohammadi:2016} and \cite[Lemma 3.1]{mohammadi:2016}, we have for $0\le i\le p-1$ and $1\le k,l\le n$,
\begin{equation}\label{eq:boundelements}
0\le p_{i,kk} \le 1 \quad \text{and} \quad -\frac{1}{2} \le p_{i,kl}\le \frac{1}{2}, \quad k\neq \ell\,.
\end{equation}

We proceed by rewriting \eqref{eq:fedfse}. To this end, we define, for $0\le i\le p-1$,
\begin{equation}\label{eq:Q}
T_i:= \tr( \A^2\E[\bfP_i] )= \sum_{k=1}^n \E[ p_{i,kk}] A_{kk} ^2 \quad \text{ and } \quad
\bfQ_i:=\frac{\A \bfP_i \A}{T_i}=(q_{i,kl})\,,
\end{equation}
such that $\E[\tr(\bfQ_i)]=T_i^{-1}\tr( \A^2\E[\bfP_i] )=1$. For future reference we remark that due to Assumption~(B) and since $\sum_{k=1}^n  p_{i,kk}=n-i$, we have
\begin{equation}\label{eq:orderTi}
C^{-1} (n-i) \le T_i \le C (n-i)\,. 
\end{equation}
Setting
\begin{equation*}
\wt Z_{i+1}:=\frac{n\,\bfu_{i+1}^{\top} \A \bfP_i \A \bfu_{i+1} -T_i}{T_i}=n\,\bfu_{i+1}^{\top} \bfQ_i \bfu_{i+1}-1\,,
\end{equation*}
we get
\begin{equation}\label{eq:segtse1}
\log \det (\bfY\bfY^{\top})= \underbrace{-p \log n+\sum_{i=0}^{p-1} \log T_i}_{=:c_n}+ \sum_{i=0}^{p-1} \log(1+\wt Z_{i+1})\,.
\end{equation}

\subsection{Properties of $\wt Z_{i+1}$} Before we continue analyzing \eqref{eq:segtse1}, we collect some essential results about the random variables $\wt Z_{i+1}$. Let $\mathcal{F}_k=\mathcal{F}_k^{(n)}$ be the sigma algebra generated by the first $k$ rows of $\bfN$. Since $\bfQ_i$ and $\bfu_{i+1}$ are independent, we see that, for $0\le i\le p-1$,
\begin{equation}\label{eq:supercool}
\bfu_{i+1}^{\top}  \bfQ_i   \bfu_{i+1} \eid \bfu_{i+1}^{\top} {\bfLambda}_{\bfQ_i} \bfu_{i+1}\,.
\end{equation}
Taking expectation, we deduce
\begin{equation*}
\begin{split}
n\,\E\left[\bfu_{i+1}^{\top}  \bfQ_i   \bfu_{i+1}\right]&= n\,\E\left[\E[\bfu_{i+1}^{\top} \bfQ_i \bfu_{i+1}\,|\,\mathcal{F}_i]\right]= n\,\underbrace{\E[U_{11}^2]}_{=1/n}
\E\left[\tr(\bfQ_i)  \right]=1\,,
\end{split}
\end{equation*}
from which we conclude that $\wt Z_{i+1}$ is centered.
We compute the variance of $\wt Z_{i+1} $  in the next lemma.
\begin{lemma}\label{lem:secondmoment}
For $0\le i\le p-1$, one has
\begin{equation}\label{eq:s2}
 \E[\wt Z_{i+1}^2]=(n^2\beta_4-1) \Big(\E[\tr({\bfLambda}_{\bfQ_i}^2)] -\frac{1}{n-1}\Big)  +{\E[\tr({\bfLambda}_{\bfQ_i}^2)](n^2 \beta_4-1)\over n-1}  +n^2\beta_{2,2} \Var(\tr(\bfQ_i))\,.
\end{equation}
\end{lemma}
\begin{proof}
In view of \eqref{eq:supercool}, we have
\begin{equation*}
 \Var(\wt Z_{i+1}^2)= n^2 \E\left[(\bfu_{i+1}^{\top} {\bfLambda}_{\bfQ_i} \bfu_{i+1})^2 \right]-1\,.
\end{equation*}
From Lemma~\ref{lem:quf}, we get by conditioning on $\mathcal{F}_i$ that
\begin{equation*}
n^2 \E\left[(\bfu_{i+1}^{\top} {\bfLambda}_{\bfQ_i} \bfu_{i+1})^2 \right]=n^2 \beta_4 \E[\tr({\bfLambda}_{\bfQ_i}^2)] +n^2\beta_{2,2}(1- \E[\tr({\bfLambda}_{\bfQ_i}^2)]) +n^2\beta_{2,2} \Var(\tr(\bfQ_i))
\end{equation*}
since $\Var(\tr(\bfQ_i))=\Var(\tr({\bfLambda}_{\bfQ_i}))=\E[(\tr({\bfLambda}_{\bfQ_i}))^2]-1$.
In conjunction with $n\beta_4+n(n-1) \beta_{2,2}=1$, which follows from taking expectation of the identity $(U_{11}^2+\cdots+U_{1n})^2=1$, the above equalities establish \eqref{eq:s2}.
\end{proof}
By our moment formula \eqref{lem:moment24}, it holds $\beta_4=3/(n(n+2))$. Now we study traces of powers of the matrices ${\bfLambda}_{\bfQ_i}$.
\begin{lemma}\label{lem:orderSj}
For $0\le i\le p-1$, we have $\norm{{\bfLambda}_{\bfQ_i}}\le T_i^{-1}C$, 
\begin{equation*}
\tr({\bfLambda}_{\bfQ_i}^2)=\frac{1}{T_i^2} \tr(\bfA^4 \bfP_i)\, \quad \text{ and } \quad \frac{C^{-j-1}}{(n-i)^{j-1}}\le \tr({\bfLambda}_{\bfQ_i}^j) \le \frac{C^{j+1}}{(n-i)^{j-1}}\,, \quad j\ge 1\,.
\end{equation*}
\end{lemma}
\begin{proof} Let $0\le i\le p-1$ and recall that $\bfQ_i=T_i^{-1} \bfA \bfP_i \bfA$, where by convention $\bfA$ is a diagonal matrix, see below \eqref{eq:dkdkdd1} for details. Hence, $\norm{{\bfLambda}_{\bfQ_i}}=T_i^{-1} \lambda_1(\bfA \bfP_i \bfA)\le T_i^{-1}C\lambda_1(\bfP_i)=T_i^{-1}C$ is immediate. Since $\bfP_i=\bfP_i^2$, we deduce that
\begin{align}\label{eq:sdsdfd1}
\bfQ_i \bfA^{-2} \bfQ_i &= \frac{1}{T_i^2} \bfA \bfP_i \bfA \bfA^{-2} \bfA \bfP_i \bfA= \frac{\bfQ_i}{T_i}\,.
\end{align}
 Using that $\bfA$ is a diagonal matrix, $\bfQ_i=\bfQ_i^{\top}$ and the fact that $\bfQ_i$ and $\bfA \bfQ_i \bfA^{-1}$ have the same eigenvalues, it follows 
\begin{align*}
\tr (\bfQ_i^2)&= \sum_{j=1}^n \lambda_j^2(\bfQ_i) =\sum_{j=1}^n \lambda_j^2(\bfA \bfQ_i \bfA^{-1})\\
&= \tr\left( (\bfA \bfQ_i \bfA^{-1}) (\bfA \bfQ_i \bfA^{-1})^{\top} \right) =
\tr(\bfA^2 \bfQ_i \bfA^{-2} \bfQ_i)\,.
\end{align*}
In combination with \eqref{eq:sdsdfd1} and since $\tr({\bfLambda}_{\bfQ_i}^2)=\tr(\bfQ_i^2)$, we obtain
\begin{equation*}
\tr({\bfLambda}_{\bfQ_i}^2)=\tr(\bfQ_i^2)=\frac{1}{T_i} \tr(\bfA^2 \bfQ_i)=\frac{1}{T_i^2} \tr(\bfA^4 \bfP_i)\,.
\end{equation*}
Now let $j\ge 1$. Since $\tr({\bfLambda}_{\bfQ_i}^j)=\sum_{k=1}^n \lambda_k^j(\bfQ_i)$ and 
$$ C^{-1} \lambda_k(\bfP_i)\le \lambda_n(\bfA^2) \lambda_k(\bfP_i)\le \lambda_k(\bfA \bfP_i \bfA)\le \lambda_1(\bfA^2) \lambda_k(\bfP_i)\le C\lambda_k(\bfP_i)$$
one gets the upper bound
\begin{equation*}
\tr({\bfLambda}_{\bfQ_i}^j)\le C T_i^{-j} \sum_{k=1}^n \lambda_k^j(\bfP_i)= \frac{C (n-i)}{T_i^{j}} \le \frac{C^{j+1}}{(n-i)^{j-1}}\,,
\end{equation*}
where \eqref{eq:orderTi} was used for the last inequality.
The derivation of the lower bound is analogous.
\end{proof}

The following result will be useful.
\begin{proposition}\label{prop:useful}
There exists a positive constant $c_0\in(0,\infty)$ such that
\begin{equation*}
\E[\wt Z_{i+1}^4] \le c_0 n^{-2}\,, \qquad 0\le i\le p-1\,.
\end{equation*}
\end{proposition}
\begin{proof}
From the definition of $\wt Z_{i+1}$ and \eqref{eq:supercool} we get for $0\le i\le p-1$,
\begin{align}\label{eq:fgdgdfs}
	&\E[(\wt Z_{i+1}-\tr \bfQ_i +1)^4]=n^4 \E\E\Big[\Big(\bfu_{i+1}^{\top} {\bfLambda}_{\bfQ_i} \bfu_{i+1}-n^{-1} \tr {\bfLambda}_{\bfQ_i} \Big)^4 \, \Big| \mathcal{F}_i \Big]\,.
  \end{align}
By Lemma \ref{lem:johnstonelemma6} and since all moments of the normal distribution are finite, we have 
\begin{equation*}
\E\Big[\Big(\bfu_{i+1}^{\top} {\bfLambda}_{\bfQ_i} \bfu_{i+1}-\frac{\tr {\bfLambda}_{\bfQ_i}}{n} \Big)^4 \, \Big| \mathcal{F}_i \Big]
\le C_4 \left[n^{-4} \big( \tr ({\bfLambda}_{\bfQ_i}^4) +( \tr ({\bfLambda}_{\bfQ_i}^2))^2 \big)+ \norm{{\bfLambda}_{\bfQ_i}}^4 n^{-2} \right],
\end{equation*}
where $C_4$ is a positive constant not depending on $n$ or $i$.
Due to $p/n\to \gamma\in (0,1)$ it holds $(1-\gamma)n \sim n-p \le n-i\le n$, so that $n-i$ is of order $n$ for all $0\le i\le p-1$, where we write $a_n\sim b_n$ for two sequences $(a_n)$ and $(b_n)$ whenever $a_n/b_n\to 1$ as $n\to\infty$. A combination of this fact with Lemma~\ref{lem:orderSj} yields that for sufficiently large $n$ there exists a positive constant $c_1$ such that $|\tr ({\bfLambda}_{\bfQ_i}^j)|\le c_1 n^{1-j}$ and $\norm{{\bfLambda}_{\bfQ_i}}\le c_1/n$ for $j\in\{1,2,3,4\}$ and $0\le i\le p-1$.	
Therefore, it follows that
\begin{align*}
\E\Big[\Big(\bfu_{i+1}^{\top} {\bfLambda}_{\bfQ_i} \bfu_{i+1}-\frac{\tr {\bfLambda}_{\bfQ_i}}{n} \Big)^4 \, \Big| \mathcal{F}_i \Big]
\le C_4 \left[n^{-4} \big( c_1 n^{-3} +c_1^2 n^{-2} \big)+ c_1^4 n^{-4} n^{-2} \right]\,
\end{align*}
and by \eqref{eq:fgdgdfs}, this establishes that there exists a constant $c_0$ such that
$\E[(\wt Z_{i+1}-\tr \bfQ_i +1)^4] \le c_0 n^{-2}$. Now we note that 
\begin{equation*}
\E[\wt Z_{i+1}^4] \le 16 \big(\E[(\wt Z_{i+1}-\tr \bfQ_i +1)^4] +\E[(\tr \bfQ_i -1)^4]\big)\,.
\end{equation*}
Since $|\tr \bfQ_i -1|= |\tr{\bfLambda}_{\bfQ_i} -1|\le c_1+1$ we get $\E[(\tr \bfQ_i -1)^4]\le (c_1+1)^2 \Var(\tr \bfQ_i) \le c_2 n^{-2}$, where the variance bound above \eqref{eq:dfhdfdfd} was used and $c_2$ is some absolute positive constant. The desired claim of the proposition follows.
\end{proof}

\begin{lemma}\label{lem:4.2}
Under the conditions of Theorem~\ref{thm:main}, we have
\begin{equation}\label{eq:maxZ}
\max_{i=0,\ldots,p-1} |\wt Z_{i+1}| \cip 0\,, \qquad \nto\,.
\end{equation}
\end{lemma}
\begin{proof}
Using the union bound, Markov's inequality and Proposition \ref{prop:useful}, we get
\begin{equation}\label{eq:sumz4}
\P\Big(\max_{i=0,\ldots,p-1} |\wt Z_{i+1}|>\vep\Big) \le \sum_{i=0}^{p-1} \P(|\wt Z_{i+1}|>\vep)\le \sum_{i=0}^{p-1}\frac{\E[\wt Z_{i+1}^4]}{\vep^4}\leq \frac{p\,c_0}{n^2\,\vep^4}\to 0 \,, \qquad n\to\infty,
\end{equation}
for any $\vep>0$.
\end{proof}

\subsection{Middlegame: Proof of Theorem~\ref{thm:main}}
Now we decompose the last term in \eqref{eq:segtse1}. By Taylor's theorem, we get
\begin{equation}\label{eq:taylornew}
\sum_{i=0}^{p-1} \log(1+\wt Z_{i+1})= \sum_{i=0}^{p-1} (\wt Z_{i+1}-\frac{\wt Z_{i+1}^2}{2}) +\sum_{i=0}^{p-1} \Omega_{i+1}\,,
\end{equation}
where the remainder in Lagrange form is given by
\begin{equation}\label{eq:Ri+1}
\Omega_{i+1}=\frac{1}{3} \Big(\frac{\wt Z_{i+1}}{1+\theta \wt Z_{i+1}} \Big)^3 \quad \text{ for some } \theta=\theta(\wt Z_{i+1})\in (0,1)\,.
\end{equation}
The Taylor expansion is justified by Lemma~\ref{lem:4.2}.
We have
\begin{equation}\label{eq:decomposecor}
\begin{split}
\sum_{i=0}^{p-1} (\wt Z_{i+1}-\frac{\wt Z_{i+1}^2}{2})&=\sum_{i=0}^{p-1} \wt Z_{i+1}-\sum_{i=0}^{p-1} \underbrace{\tfrac{1}{2} (\wt Z_{i+1}^2-\E[\wt Z_{i+1}^2 | \mathcal{F}_i])}_{=:\wt Y_{i+1}} - \sum_{i=0}^{p-1}\tfrac{1}{2}\E[\wt Z_{i+1}^2 | \mathcal{F}_i]\,.
\end{split}
\end{equation}
Our next goal is to show
\begin{equation}\label{eq:mainintermediate}
\frac{\log \det(\bfY\bfY^{\top}) -\wt\mu_n}{\wt\sigma_n} \cid N(0,1)\,, \qquad \nto\,,
\end{equation}
where, for $n\ge 1$, $\wt \sigma_n^2:=  \sum_{i=0}^{p-1}\E[\wt Z_{i+1}^2]$  and the centering sequence $\wt \mu_n$ is given by
\begin{equation}\label{eq:mean1}
\wt \mu_n:=p \log \frac{\tr(\A^{\top} \A)}{n} -\frac{\wt\sigma_n^2}{2}+ \sum_{i=0}^{p-1} \log T_i -p\log n=-\frac{\wt\sigma_n^2}{2}+ \sum_{i=0}^{p-1} \log T_i -p\log n\,.
\end{equation}
In view of \eqref{eq:segtse1}, \eqref{eq:taylornew} and \eqref{eq:decomposecor}, one gets
\begin{equation}\label{eq:dddd}
\log \det (\bfY\bfY^{\top}) -\wt\mu_n = 
\sum_{i=0}^{p-1} \wt Z_{i+1}-\sum_{i=0}^{p-1} \wt Y_{i+1}+\sum_{i=0}^{p-1} \Omega_{i+1}
  - \sum_{i=0}^{p-1}\tfrac{1}{2}\E[\wt Z_{i+1}^2 | \mathcal{F}_i]+c_n-\wt\mu_n.
\end{equation}

By virtue of \eqref{eq:dddd}, distributional convergence  \eqref{eq:mainintermediate} follows from the next four limit relations by an application of the Slutsky lemma,

\begin{align}
\frac{1}{\wt\sigma_n} \sum_{i=0}^{p-1} \wt Z_{i+1} &\cid N(0,1)\,, \label{sumZ}\\
\sum_{i=0}^{p-1} \wt Y_{i+1}&\cip 0\,, \label{sumY_i}\\
\sum_{i=0}^{p-1} \Omega_{i+1}&\cip 0\,, \label{sumR_i}\\
\sum_{i=0}^{p-1} \tfrac{1}{2}\E[\wt Z_{i+1}^2 | \mathcal{F}_i]-c_n+\wt\mu_n&\cip 0\,, \label{sumconst}
\end{align}
as $\nto$. Note that, by Lemma~\ref{lem:sigma}, $\wt\sigma_n^2$ is of constant order. For this reason we omitted $\wt\sigma_n^{-1}$ in \eqref{sumY_i}, \eqref{sumR_i} and \eqref{sumconst}. Equations \eqref{sumZ}, \eqref{sumY_i}, \eqref{sumR_i}, \eqref{sumconst} are proved in Sections \ref{sec:sumZ}, \ref{sec:sumY_i}, \ref{sec:sumR_i} and \ref{sec:sumconst}, respectively. This establishes \eqref{eq:mainintermediate}. 

\subsection{Endgame: Fine-tuning the norming sequences} 
In this subsection we complete the proof of Theorem~\ref{thm:main} by providing simpler formulas for the mean and variance.  We start by showing that the sequences $\wt\mu_n$ and $\wt\sigma_n^2$ in \eqref{eq:mainintermediate} can be replaced by
\begin{equation}\label{defmu}
\mu_n:=p \log \frac{\tr(\A^{\top} \A)}{n} -\frac{\sigma_n^2}{2}+ \sum_{i=1}^{p-1} \log \frac{\tr(\A^2 \E[\bfP_i])}{n} 
\end{equation}
and
\begin{equation}\label{defsigma}
\sigma_n^2:=-2\frac{p}{n} + 2 \sum_{i=1}^{p-1} \frac{\tr(\A^4 \E[\bfP_i])}{(\tr(\A^2 \E[\bfP_i]))^2}\,,
\end{equation}
respectively. Here, the summations start at $i=1$ since $\bfP_0=\bfI_n$ and $T_0=n$. The dependence on $\bfA$ of these new sequences is more explicit. Recall that by our notational convention $\bfA$ is a diagonal matrix, see below \eqref{eq:dkdkdd1} for details.

\begin{lemma}\label{lem:sigma}
It holds $\sigma_n^2 \sim \wt\sigma_n^2$, as $\nto$, and 
\begin{equation*}
-2\frac{p}{n} + 2 \sum_{i=0}^{p-1} \frac{1}{n-i}\le \sigma_n^2  \le -2\frac{p}{n} + 2C^4 \sum_{i=0}^{p-1} \frac{1}{n-i}  \,,
\end{equation*}
\begin{equation}\label{eq:var1}
 -2\frac{p}{n} + 2 \sum_{i=0}^{p-1} \frac{1}{n-i} \sim -2\frac{p}{n} -2\log\Big(1-\frac{p}{n}\Big)\to -2 \gamma -2\log(1-\gamma)>0\,,\qquad \nto\,.
\end{equation}
\end{lemma}
\begin{proof} 
From Lemma~\ref{lem:secondmoment}, Lemma~\ref{lem:orderSj}, \eqref{eq:dfhdfdfd} and the facts that $n^2 \beta_4 \to 3$ and $n^2 \beta_{2,2} \to 1$, we get
\begin{align*}
\wt \sigma_n^2&=  \sum_{i=0}^{p-1} \Bigg( (n^2\beta_4-1) \Big(\E[\tr({\bfLambda}_{\bfQ_i}^2)] -\frac{1}{n-1}\Big)  +{\E[\tr({\bfLambda}_{\bfQ_i}^2)](n^2 \beta_4-1)\over n-1}
 +n^2\beta_{2,2} \Var(\tr(\bfQ_i))\Bigg)\\
&=  -2 \frac{p}{n} +2 \sum_{i=0}^{p-1} \frac{\tr(\A^4 \E[\bfP_i])}{(\tr(\A^2 \E[\bfP_i]))^2} +o(1)\,, \qquad \nto\,.
\end{align*}
The fact that $\sum_{i=0}^{p-1} \frac{1}{n-i} \sim -\log(1-p/n)$ follows from asymptotic properties of the harmonic series, more precisely, by using that $\sum_{k=1}^n 1/k -\log n$ tends to the Euler--Mascheroni constant.
The last limit in \eqref{eq:var1} follows from $p/n\to \gamma\in (0,1)$. To complete the proof of the lemma, it suffices to show  
$$\frac{1}{n-i} \le \frac{\tr(\A^4 \E[\bfP_i])}{(\tr(\A^2 \E[\bfP_i]))^2}\le \frac{C^4}{n-i} .$$
For $\ell\in\{2,4\}$ and conditionally on $\mathcal{F}_i$, define the random variables $\eta_{\ell}$ through
$$\P(\eta_{\ell}=A_{kk})=\frac{p_{i,kk}}{n-i}\,, \qquad k=1,\ldots,n\,,$$
where we recall that $\tr \bfP_i=n-i$. By Jensen's inequality, we deduce
$$ (n-i) \tr(\A^4 \E[\bfP_i]) = (n-i)^2 \E[\eta_{\ell}^4] \ge (n-i)^2 \big(\E[\eta_{\ell}^2]\big)^2 = (\tr(\A^2 \E[\bfP_i]))^2\,.$$
Together with $C^{-1}\le A_{nn}^2 \le A_{11}^2\le C$ this yields the desired claim.
\end{proof}
An immediate consequence of Lemma~\ref{lem:sigma} is $\wt\mu_n-\mu_n\to 0$. Next, we observe that only the expectations of the diagonal elements of $\bfP_i$ are needed in \eqref{defmu} and \eqref{defsigma}. Indeed, we have
$$\tr(\A^2 \E[\bfP_i])=\sum_{k=1}^n A_{kk}^2 \E[p_{i,kk}] \quad \text{ and } \quad \tr(\A^4 \E[\bfP_i])=\sum_{k=1}^n A_{kk}^4 \E[p_{i,kk}]\,.$$
Thus our final goal is to find $t_{i,k}(\A):= \E[p_{i,kk}]$ for $1\le i\le p-1$ and $1\le k\le n$.

\begin{lemma}\label{lem:tik}
For $1\le i\le p-1$ and $1\le k\le n$, let $\bfw_{i1}, \ldots, \bfw_{in}$ be i.i.d.\ $i$-dimensional random vectors whose components are independent standard normal random variables. Then it holds 
\begin{equation}\label{eq:tikdag}
t_{i,k}(\A)= \E\left[\frac{1}{1+ A_{kk}^2 \bfw_{ik}^{\top} \big(\sum_{\ell=1; \ell \neq k}^n A_{\ell\ell}^2 \bfw_{i\ell} \bfw_{i\ell}^{\top}\big)^{-1}\bfw_{ik}}\right]\,.
\end{equation}
\end{lemma}
\begin{proof}
Denote the $k$-th column of $\bfN_{(i)}$ by $\bfw_{ik}$ and write $\bfN_{(i,k)}$ for the matrix obtained from $\bfN_{(i)}$ by removing its $k$-th column. Let $\bfA_{(k)}$ be the matrix $\A$ without its $k$-th row and column.
Using \eqref{eq:repPi} and the Sherman-Morrison formula \cite{SM50}, we get
\begin{equation}\label{eq:pikk}
\begin{split}
p_{i,kk}
&= 1- A_{kk}^2 \bfw_{ik}^{\top} \big(\bfN_{(i)}\bfA^2 \bfN_{(i)}^{\top} \big)^{-1}\bfw_{ik}\\
&= 1- A_{kk}^2 \bfw_{ik}^{\top} \big(\bfN_{(i,k)}\bfA_{(k)}^2 \bfN_{(i,k)}^{\top} +A_{kk}^2 \bfw_{ik} \bfw_{ik}^{\top}\big)^{-1}\bfw_{ik}\\
&= \frac{1}{1+ A_{kk}^2 \bfw_{ik}^{\top} \big(\bfN_{(i,k)}\bfA_{(k)}^2 \bfN_{(i,k)}^{\top} \big)^{-1}\bfw_{ik}}\,.
\end{split}
\end{equation}
Noting that the entries of $\bfN_{(i)}$ are independent standard normal random variables finishes the proof.
\end{proof}
\begin{remark}\label{rem:tik}{\em
If $\A=\bfI_n$, one has $t_{i,k}(\bfI_n)=(n-i)/n$ since $\tr \bfP_i=n-i$ and $p_{i,11},\ldots, p_{i,nn}$ are identically distributed. We stress that the formula for $t_{i,k}(\A)$ in \eqref{eq:tikdag} is valid for diagonal matrices $\A$ and observe
\begin{equation}\label{eq:tricks}
t_{i,k}(\A)= t_{i,k}(c \,\A)\,, \qquad c\in \R\backslash\{0\}\,.
\end{equation} 
}\end{remark}
For $s\in\{2,4\}$ we get
$\tr(\A^s \E[\bfP_i])=\sum_{k=1}^n A_{kk}^s t_{i,k}(\A)$
and therefore, we may write \eqref{defmu} and \eqref{defsigma} as follows:
\begin{align*}
\mu_n&=p \log \frac{\tr(\A^{\top} \A)}{n} -\frac{\sigma_n^2}{2}+ \sum_{i=1}^{p-1} \log \frac{\sum_{k=1}^n A_{kk}^2 t_{i,k}(\A)}{n}\,,\\
\sigma_n^2&=-2\frac{p}{n} + 2 \sum_{i=1}^{p-1} \frac{\sum_{k=1}^n A_{kk}^4 t_{i,k}(\A)}{(\sum_{k=1}^n A_{kk}^2 t_{i,k}(\A))^2}\,.
\end{align*}

So far we have assumed that $\A$ is a diagonal matrix with positive diagonal elements $A_{jj}=\sqrt{\lambda_j(\A^{\top} \A)}$, $j=1,\ldots,n$ and $\tr(\A^{\top} \A)=n$. For the general case we need to replace $t_{i,k}(\A)$ with $t_{i,k}(\wt\A)$, where 
$$\wt\A:= \sqrt{\frac{n}{\tr(\A^{\top}\A)}} \, {\bfLambda}_{\A^{\top}\A}^{1/2}\,,$$
with $k$-th diagonal element of $\wt\A^2$ given by
$\frac{n}{\tr(\A^{\top}\A)} \lambda_k(\A^{\top}\A)$. In view of \eqref{eq:tricks}, we have $t_{i,k}(\wt\A)=t_{i,k}\big( {\bfLambda}_{\A^{\top}\A}^{1/2} \big)$. To unify notation, we define for (not necessarily diagonal) matrices $\A$, 
\begin{equation}\label{eq:tikdag1}
t_{i,k}(\A)= \E\left[\frac{1}{1+ \lambda_k(\A^{\top}\A) \bfw_{ik}^{\top} \big(\sum_{\ell=1; \ell \neq k}^n \lambda_{\ell}(\A^{\top}\A) \bfw_{i\ell} \bfw_{i\ell}^{\top}\big)^{-1}\bfw_{ik}}\right]\,,
\end{equation}
which coincides with \eqref{eq:tikdag} in case $\A$ is a diagonal matrix. 

The above considerations establish that the sequences $\wt\mu_n$ and $\wt\sigma_n^2$ in \eqref{eq:mainintermediate} can be replaced by
\begin{align*}
\mu_n&=p \log \frac{\tr(\A^{\top} \A)}{n} -\frac{\sigma_n^2}{2}+ \sum_{i=1}^{p-1} \log \frac{\sum_{k=1}^n \frac{n}{\tr(\A^{\top}\A)}\lambda_k(\A^{\top}\A) t_{i,k}(\A)}{n}\\
&= \log \tr(\A^{\top}\A) -p \log n -\frac{\sigma_n^2}{2}+ \sum_{i=1}^{p-1} \log \bigg(\sum_{k=1}^n \lambda_k(\A^{\top}\A) t_{i,k}(\A)\bigg)\,,\\
\sigma_n^2&=-2\frac{p}{n} + 2 \sum_{i=1}^{p-1} \frac{\sum_{k=1}^n \lambda_k^2(\A^{\top}\A) t_{i,k}(\A)}{(\sum_{k=1}^n \lambda_k (\A^{\top}\A) t_{i,k}(\A))^2}\,,
\end{align*}
respectively, which completes the proof of Theorem~\ref{thm:main}.

\subsection{Proof of \eqref{sumY_i}}\label{sec:sumY_i}

By Markov's inequality, one has for $\vep>0$,
\begin{equation}\label{eq:boundY}
\P\Big(\Big|\sum_{i=0}^{p-1} \wt Y_{i+1}\Big|>\vep\Big)\le \vep^{-2} \E\Big[\Big(\sum_{i=0}^{p-1} \wt Y_{i+1}\Big)^2 \Big]\,.
\end{equation}
If $j\neq i$ one can show by conditioning on $\mathcal{F}_{\max(i,j)}$ that 
$\E[\wt Y_{i+1} \wt Y_{j+1}]=0$. Therefore, one gets
\begin{equation*}
\begin{split}
\E\Big[\Big(\sum_{i=0}^{p-1} \wt Y_{i+1}\Big)^2 \Big]&= \sum_{i=0}^{p-1} \E[\wt Y_{i+1}^2]= \frac14 \sum_{i=0}^{p-1} \E\Big[(\wt Z_{i+1}^2-\E[\wt Z_{i+1}^2 | \mathcal{F}_i])^2\Big]\\
&\le \frac12 \sum_{i=0}^{p-1}  \E[\wt Z_{i+1}^4] 
+ \frac12 \sum_{i=0}^{p-1} \E[(\E[\wt Z_{i+1}^2 | \mathcal{F}_i])^2] \,,
\end{split}
\end{equation*}
where the first term in the last line is $o(1)$ by Proposition \ref{prop:useful}.
Using Lemma~\ref{lem:secondmoment}, Lemma~\ref{lem:orderSj} and \eqref{lem:moment24}, we see that, uniformly in $i$,  
\begin{align*}
\E[\wt Z_{i+1}^2 | \mathcal{F}_i] -n^2\beta_{2,2} (\tr(\bfQ_i)-1)^2&=
(n^2\beta_4-1) \Big(\tr({\bfLambda}_{\bfQ_i}^2) -\frac{1}{n-1}\Big) +{\tr({\bfLambda}_{\bfQ_i}^2)(n^2 \beta_4-1)\over n-1} =O(n^{-1})\,.
\end{align*}
Therefore we have by \eqref{eq:dfhdfdfd} and since $\tr(\bfQ_i)+1$ is uniformly bounded by some constant $c$,
\begin{align*}
\sum_{i=0}^{p-1}\E[(\E[\wt Z_{i+1}^2 | \mathcal{F}_i])^2]&\le O(n^{-1})+2 (n^2\beta_{2,2})^2 \sum_{i=0}^{p-1} \E\big[\big((\tr(\bfQ_i))^2-1\big)^2\big]\\
&\le O(n^{-1})+2  \sum_{i=0}^{p-1}\E\big[(\tr(\bfQ_i)-1)^2(\tr(\bfQ_i)+1)^2\big]\\
&\le O(n^{-1})+2 c^2 \sum_{i=0}^{p-1}\E\big[(\tr(\bfQ_i)-1)^2\big]=o(1)\,.
\end{align*}
We conclude 
$$\lim_{\nto} \E\Big[\Big(\sum_{i=0}^{p-1} \wt Y_{i+1}\Big)^2 \Big] =0\,.$$
In view of \eqref{eq:boundY}, we have proved \eqref{sumY_i}.

\subsection{Proof of \eqref{sumZ}} \label{sec:sumZ}
Set $\bar{Z}_{i+1} := \wt Z_{i+1}-\E[\wt Z_{i+1} | \mathcal{F}_i]=\wt Z_{i+1}- (\tr \bfQ_i -1)$. By Lemma \ref{lem:yaskov}, we have $\wt\sigma_n^{-1} \sum_{i=0}^{p-1}(\tr \bfQ_i -1)=o_{\P}(1)$, as $\nto$, and thus it suffices to show  
\begin{align}\label{sumZ11}
\frac{1}{\wt\sigma_n} \sum_{i=0}^{p-1} \bar{Z}_{i+1} &\cid N(0,1)\,. 
\end{align}
To this end, we will use the following CLT for martingale differences.
\begin{lemma}[e.g. Hall and Heyde \cite{hall:heyde:1980}, Theorem~3.2] \label{lem:martingaleclt}
Let $\{S_{ni},\mathcal{F}_{ni}, 1\le i\le k_n, n\ge 1\}$ be a zero-mean, square integrable martingale array with differences $Z_{ni}$. Suppose that $\E[\max_i Z_{ni}^2]$ is bounded in $n$ and that
 \begin{equation*}
\max_i |Z_{ni}|\cip 0 \quad \text{ and } \quad \sum_i Z_{ni}^2 \cip 1\,.
\end{equation*}
Then we have $S_{nk_n}\cid N(0,1)$ as $n\to\infty$.
\end{lemma}
In view of $\E[\bar{Z}_{i+1} | \mathcal{F}_i]=0$, we observe that $(\bar{Z}_{i+1})_i$ is a martingale difference sequence with respect to the filtration $(\mathcal{F}_i)$.
We apply Lemma~\ref{lem:martingaleclt} to the martingale differences $\wt\sigma_n^{-1} \bar{Z}_{i+1}$.
From \eqref{eq:maxZ} and \eqref{eq:fdfdd1}, we have 
$$\max_{i=0,\ldots,p-1} | \wt\sigma_n^{-1} \,\bar{Z}_{i+1}| \le \wt\sigma_n^{-1} \Big( \max_{i=0,\ldots,p-1} | \wt Z_{i+1}|+ \max_{i=0,\ldots,p-1} | \tr \bfQ_i -1|\Big)\cip 0\,, \qquad \nto\,.$$
Next, we see that
\begin{equation*}
\wt\sigma_n^{-2} \E\Big[\max_{i=0,\ldots,p-1} \bar{Z}_{i+1}^2\Big]\le 2\,\wt\sigma_n^{-2} \sum_{i=0}^{p-1} \Big(\E[\wt Z_{i+1}^2] +\Var\big(\tr(\bfQ_i)\big) \Big)=2+o(1)\,,
\end{equation*}
where the last equality follows from the definition of $\wt\sigma_n^2$ and \eqref{eq:dfhdfdfd}. 

Using Markov's inequality and \eqref{eq:dfhdfdfd}, it can be checked that $\sum_{i=0}^{p-1}\big( \bar{Z}_{i+1}^2-\wt Z_{i+1}^2\big) \cip 0$. 
Due to $\wt\sigma_n^{-2} \sum_{i=0}^{p-1}\E[\wt Z_{i+1}^2]=1$, the condition $\wt\sigma_n^{-2} \sum_{i=0}^{p-1}\bar{Z}_{i+1}^2\cip 1$ is therefore implied by
\begin{equation}\label{eq:1212}
 \sum_{i=0}^{p-1}(\wt Z_{i+1}^2-\E[\wt Z_{i+1}^2\,| \mathcal{F}_i])\cip 0 \,,\qquad \nto\,,
\end{equation}
and 
\begin{equation}\label{eq:reess}
 \sum_{i=0}^{p-1}(\E[\wt Z_{i+1}^2\,| \mathcal{F}_i]-\E[\wt Z_{i+1}^2]) \cip 0 \,,\qquad \nto\,.
\end{equation}
Observe that \eqref{eq:1212} is equivalent to \eqref{sumY_i}. Hence, it remains to show \eqref{eq:reess}. To this end, recall that in Lemma \ref{lem:secondmoment} and its proof we obtained
\begin{align*}
\sum_{i=0}^{p-1}\big(\E[\wt Z_{i+1}^2\,| \mathcal{F}_i]-\E[\wt Z_{i+1}^2]\big) &=
\Big(n^2\beta_4-1+{(n^2 \beta_4-1)\over n-1}\Big) \sum_{i=0}^{p-1} \Big(\tr({\bfLambda}_{\bfQ_i}^2)-\E[\tr({\bfLambda}_{\bfQ_i}^2)]\Big) \\
&\quad +n^2\beta_{2,2} \sum_{i=0}^{p-1}\Big( (\tr(\bfQ_i))^2-\E\big[ (\tr(\bfQ_i))^2\big]\Big)=:S^{(1)}+S^{(2)}\,.
\end{align*}
Using $n^2 \beta_4\to 3$, $n^2 \beta_{2,2}\to 1$ and Lemma~\ref{lem:yaskov}, we get
\begin{align*}
S^{(1)}&\sim 2 \sum_{i=0}^{p-1} \Big(\tr({\bfLambda}_{\bfQ_i}^2)-\E[\tr({\bfLambda}_{\bfQ_i}^2)]\Big)\cip 0 \quad \text{ and } \quad S^{(2)}\cip 0\,, \qquad \nto\,.
\end{align*}
Thus, we have verified the conditions of Lemma~\ref{lem:martingaleclt} which now yields \eqref{sumZ} and finishes the proof.

\subsection{Proof of \eqref{sumR_i}}\label{sec:sumR_i}

Set $\delta=1/2$ and define the event $E_n(\delta)=\{\max_{i=0,\ldots,p-1} |\wt Z_{i+1}|\le \delta \}$. For any $\vep>0$, it follows that
\begin{equation*}
\P\left( \Big|\sum_{i=0}^{p-1} \Omega_{i+1}\Big| >\vep\right)\le \P\left( \Big|\sum_{i=0}^{p-1} \Omega_{i+1}\Big| \1_{E_n(\delta)}>\vep\right)+\P\left(\max_{i=0,\ldots,p-1} |\wt Z_{i+1}|>\delta\right)\,.
\end{equation*}
The second term on the right-hand side tends to zero by virtue of \eqref{eq:maxZ}. 
From \eqref{eq:Ri+1}, recall that 
\begin{equation*}
\Omega_{i+1}=\frac{1}{3} \Big(\frac{\wt Z_{i+1}}{1+\theta \wt Z_{i+1}} \Big)^3 \quad \text{ for some } \theta=\theta(\wt Z_{i+1})\in (0,1)\,.
\end{equation*}
On the event $E_n(\delta)$ we have $1+\theta \wt Z_{i+1}\in (1-\delta,1+\delta)$ and therefore
\begin{equation*}
\begin{split}
\P\left( \Big|\sum_{i=0}^{p-1} \Omega_{i+1}\Big| \1_{E_n(\delta)}>\vep\right)
&\le \frac{1}{3\vep} \sum_{i=0}^{p-1}  \E\Big[\Big|\frac{\wt Z_{i+1}}{1+\theta \wt Z_{i+1}}\Big|^3 \1_{E_n(\delta)}\Big]\\
&\le \frac{1}{3\vep} \sum_{i=0}^{p-1} (1-\delta)^{-3}  \E\Big[|\wt Z_{i+1}|^3 \1_{E_n(\delta)}\Big]\\
&\le \frac{1}{3\vep (1-\delta)^{3}} \sum_{i=0}^{p-1}   \Big(\E\Big[|\wt Z_{i+1}|^4\Big]\Big)^{3/4}\\
&\le \frac{1}{3\vep (1-\delta)^{3}} p \big(c_0n^{-2}\big)^{3/4}\to 0\,,\qquad \nto\,,
\end{split}
\end{equation*}
where Hölder's inequality with $p=4/3$ and $q=4$ was used for the third inequality and Proposition~\ref{prop:useful} for the last.

\subsection{Proof of \eqref{sumconst}}\label{sec:sumconst}

In view of \eqref{eq:reess}, equation \eqref{sumconst} follows from
\begin{equation*}
\lim_{\nto}\Big[ \sum_{i=0}^{p-1} \tfrac{1}{2}\E[\wt Z_{i+1}^2] -c_n+\wt\mu_n\Big]=\lim_{\nto}\big[ \wt\sigma_n^2/2-c_n+\wt\mu_n\big]=0\,.
\end{equation*}

\appendix

\section{A technical lemma}\setcounter{equation}{0}

{Throughout this section we will assume the conditions of Theorem~\ref{thm:main} and use the notation from Section~\ref{sec:proofmain}. As in Section~\ref{sec:proofmain}, we will assume that $\A$ is a diagonal matrix with positive diagonal elements $A_{jj}=\sqrt{\lambda_j(\A^{\top} \A)}$, $j=1,\ldots,n$ and $\tr(\A^{\top} \A)=n$. By Assumption (B), there exists a constant $C\ge 1$ not depending on $n$ such that $C^{-1}\le A_{nn}^2 \le A_{11}^2\le C$.}

\begin{lemma}\label{lem:yaskov}
It holds, as $\nto$, 
\begin{align}
&\sum_{i=0}^{p-1} \Big(\tr({\bfLambda}_{\bfQ_i}^2)-\E[\tr({\bfLambda}_{\bfQ_i}^2)]\Big) \cip 0\,, \label{eq:fdfdd}\\
&\sum_{i=0}^{p-1} \big(\tr(\bfQ_i)-1\big) \cip 0 \qquad \text{ and } \qquad \max_{i=0,\ldots,p-1} \big|\tr(\bfQ_i)-1\big| \cip 0\,,\label{eq:fdfdd1}\\
& \sum_{i=0}^{p-1}\Big( (\tr(\bfQ_i))^2-\E\big[ (\tr(\bfQ_i))^2\big]\Big)\cip 0\,,\label{eq:lastdirt}
\end{align}
where the matrix $\bfQ_i$ is defined in \eqref{eq:Q} and ${\bfLambda}_{\bfQ_i}$ is defined by \eqref{eq:SpectralDecomposition}. Moreover, it holds for $\vep>0$ that
\begin{equation}\label{eq:dfhdfdfd}
\P\Big(\max_{i=0,\ldots,p-1} \big|\tr(\bfQ_i)-1\big|>\vep\Big)\le \frac{1}{\vep^2}  \sum_{i=1}^{p-1}\Var\big(\tr(\bfQ_i)\big) = O(1/n)\,, \qquad \nto\,.
\end{equation}
\end{lemma}

\begin{proof}
Let us first assume that 
\begin{equation}\label{eq:nastypart}
\max_{1\le i\le p-1} \max_{1\le k\le n} \E\big[(p_{i,kk}-\E[p_{i,kk}])^2\big] =O(1/n)\,, \qquad \nto\,.
\end{equation}
From Lemma~\ref{lem:orderSj} we know that 
\begin{align*}
\tr({\bfLambda}_{\bfQ_i}^2)-\E[\tr({\bfLambda}_{\bfQ_i}^2)])&= \frac{1}{T_i^2} \tr(\bfA^4 (\bfP_i-\E[\bfP_i]))\\
&= \frac{1}{T_i^2} \sum_{k=1}^n A_{kk}^4 (p_{i,kk}-\E[p_{i,kk}]),
\end{align*}
and by \eqref{eq:orderTi} we also have $C^{-1} (n-i) \le T_i \le C (n-i)$. In view of $\bfP_0=\bfI_n$, we may start the summation in \eqref{eq:fdfdd} (and \eqref{eq:fdfdd1}) at $i=1$. In what follows, the notation $a_n\lesssim b_n$ means $a_n\le c\, b_n$ for $n\in \N$ and some positive constant $c$ that does not depend on $i$ or $k$. Applications of Markov's and Lyapunov's inequalities yield for $\eta>0$,
\begin{align*}
\P&\Big(\Big|\sum_{i=1}^{p-1} \frac{1}{T_i^2} \sum_{k=1}^n A_{kk}^4 (p_{i,kk}-\E[p_{i,kk}])\Big| >\eta\Big) 
\le \frac{1}{\eta} \sum_{i=1}^{p-1} \frac{1}{T_i^2} \sum_{k=1}^n A_{kk}^4 \E\big[|p_{i,kk}-\E[p_{i,kk}]|\big]\\
&\lesssim \frac{1}{n^2 \eta} \sum_{i=1}^{p-1} \sum_{k=1}^n A_{kk}^4\sqrt{\E\big[(p_{i,kk}-\E[p_{i,kk}])^2\big] }
\lesssim \frac{1}{\eta} \max_{1\le i\le p-1} \max_{1\le k\le n} \sqrt{\E\big[(p_{i,kk}-\E[p_{i,kk}])^2\big] }
\end{align*}
and the right-hand side converges to zero by \eqref{eq:nastypart}. This establishes \eqref{eq:fdfdd}. 

Next, we turn to \eqref{eq:fdfdd1}. Recall from \eqref{eq:Q} that $\E[\tr(\bfQ_i)]=1$ and 
\begin{equation}\label{eq:trqi}
\tr(\bfQ_i)=\frac{1}{T_i} \sum_{k=1}^n A_{kk}^2 p_{i,kk}\,.
\end{equation}
Since $\tr(\bfP_i)=n-i$ is nonrandom and setting $a_k:= A_{kk}^2-\tr(\A^2)/n$, we have
\begin{align*}
\E\bigg[\Big(\sum_{i=1}^{p-1} (\tr(\bfQ_i)-1)\Big)^2\bigg]&= \Var\bigg(\sum_{i=1}^{p-1} \tr(\bfQ_i)\bigg)\\
&= \Var\bigg(\sum_{i=1}^{p-1} \frac{1}{T_i} \sum_{k=1}^n a_k \, p_{i,kk}\bigg) = \Var\bigg(\sum_{k=1}^n a_k S_k\bigg)\\
&\le \bigg(\sum_{k=1}^n a_k^2\bigg) \bigg(\sum_{k,\ell=1}^n \Cov^2(S_k,S_\ell)\bigg)^{1/2}\,,
\end{align*}
where the Cauchy-Schwarz inequality was used and $S_k:=\sum_{i=1}^{p-1} T_i^{-1} p_{i,kk}$.
By \eqref{eq:nastypart} and \eqref{eq:orderTi} we get for $k,\ell=1,\ldots,n$,
$$|\Cov(S_k,S_\ell)|\le \sum_{i,j=1}^{p-1} \frac{1}{T_iT_j} |\Cov(p_{i,kk},p_{i,\ell\ell})| \lesssim \sum_{i,j=1}^{p-1} n^{-2} n^{-1} \lesssim n^{-1}\,.$$
Therefore, we conclude that 
\begin{align*}
\E\bigg[\Big(\sum_{i=1}^{p-1} (\tr(\bfQ_i)-1)\Big)^2\bigg]
&\lesssim \sum_{k=1}^n a_k^2 
= \tr\Big( \A^2 - \frac{\tr(\A^2)}{n} \bfI_n\Big)^2 
\end{align*}
and the right-hand side\ tends to zero by assumption \eqref{eq:B2}, establishing the first convergence in \eqref{eq:fdfdd1}. For the second one similar considerations yield for $i=1,\ldots, p-1$ that
\begin{align*}
\Var\big(\tr(\bfQ_i)\big)
&= \frac{1}{T_i^2}\Var\bigg( \sum_{k=1}^n a_k \, p_{i,kk}\bigg) \\
&\le \frac{1}{T_i^2} \bigg(\sum_{k=1}^n a_k^2\bigg) \bigg(\sum_{k,\ell=1}^n \Cov^2(p_{i,kk},p_{i,\ell\ell})\bigg)^{1/2}\\
&\lesssim \frac{1}{n^2} \sum_{k=1}^n a_k^2 \lesssim \frac{1}{n^2}\,.
\end{align*}
In combination with the union bound and Markov's inequality this shows for any $\vep>0$ that
\begin{equation*}
\P\Big(\max_{i=0,\ldots,p-1} \big|\tr(\bfQ_i)-1\big|>\vep\Big)\le \frac{1}{\vep^2}  \sum_{i=1}^{p-1}\Var\big(\tr(\bfQ_i)\big) \lesssim \frac{1}{n\,\vep^2}
\end{equation*}
and so \eqref{eq:dfhdfdfd} and the second convergence in \eqref{eq:fdfdd1} hold.

Regarding \eqref{eq:lastdirt}, we get from \eqref{eq:trqi} that
\begin{align*}
(\tr(\bfQ_i))^2-\E\big[ (\tr(\bfQ_i))^2\big] &= \frac{1}{T_i^2} \sum_{k,\ell=1}^n A_{kk}^2A_{\ell \ell}^2 \big(p_{i,kk}p_{i,\ell \ell}-\E[p_{i,kk}p_{i,\ell \ell}]\big)\,.
\end{align*}
Using the identity $\tr(\bfP_i)=n-i$ and the definition of $\bfQ_i$, we have
\begin{align*}
&\frac{1}{T_i^2}\sum_{k,\ell=1}^n a_k a_{\ell} \big(p_{i,kk}p_{i,\ell \ell}-\E[p_{i,kk}p_{i,\ell \ell}]\big)\\
&=\frac{1}{T_i^2}\sum_{k,\ell=1}^n A_{kk}^2A_{\ell \ell}^2 \big(p_{i,kk}p_{i,\ell \ell}-\E[p_{i,kk}p_{i,\ell \ell}]\big) -2\frac{\tr(\A^2)}{n}\frac{(n-i)}{T_i^2} \sum_{k=1}^n A_{kk}^2 \big(p_{i,kk}-\E[p_{i,kk}]\big)\\
&= (\tr(\bfQ_i))^2-\E\big[ (\tr(\bfQ_i))^2 \big]
-2\frac{\tr(\A^2)}{n}\frac{(n-i)}{T_i} (\tr(\bfQ_i)-1)
\end{align*}
and therefore,
\begin{align}
&\sum_{i=1}^{p-1}\Big( (\tr(\bfQ_i))^2-\E\big[ (\tr(\bfQ_i))^2\big]\Big) \nonumber\\
&=\sum_{i=1}^{p-1}\frac{1}{T_i^2}\sum_{k,\ell=1}^n a_k a_{\ell} \big(p_{i,kk}p_{i,\ell \ell}-\E[p_{i,kk}p_{i,\ell \ell}]\big) 
+2\frac{\tr(\A^2)}{n} \sum_{i=1}^{p-1}\frac{(n-i)}{T_i} (\tr(\bfQ_i)-1)\,. \label{eq:dsfas}
\end{align}
Since $\tr(\A^2)(n-i)/(n T_i) \lesssim 1$, the same arguments as in the proof of the first convergence in \eqref{eq:fdfdd1} yield that the second term in \eqref{eq:dsfas} tends to zero in probability as $\nto$. 
The expectation of the absolute value of the first term in \eqref{eq:dsfas} is bounded by 
\begin{align*}
&\sum_{i=1}^{p-1}\frac{1}{T_i^2}\sum_{k,\ell=1}^n |a_k a_{\ell}| \E\big[ \big|p_{i,kk}p_{i,\ell \ell}-\E[p_{i,kk}p_{i,\ell \ell}]\big| \big]\\
&\le \sum_{i=1}^{p-1}\frac{1}{T_i^2} \bigg(\sum_{k=1}^n a_k^2\bigg) \bigg(\sum_{k,\ell=1}^n \E\big[ \big|p_{i,kk}p_{i,\ell \ell}-\E[p_{i,kk}p_{i,\ell \ell}]\big| \big]^2\bigg)^{1/2}\lesssim \sum_{i=1}^{p-1} n^{-2} \bigg(\sum_{k=1}^n a_k^2\bigg) n \to 0\,,
\end{align*}
as $\nto$, where we used Cauchy-Schwarz, $\big|p_{i,kk}p_{i,\ell \ell}-\E[p_{i,kk}p_{i,\ell \ell}]\big|\le 2$, see \eqref{eq:boundelements}, and assumption \eqref{eq:B2}. Thus, by Markov's inequality we obtain that the first term in \eqref{eq:dsfas} tends to zero in probability as $\nto$, which concludes the proof of \eqref{eq:lastdirt}.
\medskip

Finally, we turn to the proof of \eqref{eq:nastypart}. Denote the $k$-th column of $\bfN_{(i)}$ by $\bfw_{ik}$
and set $\bfB_{(i)}=\bfN_{(i)} \bfA$. Then the $k$-th column of $\bfB_{(i)}$ is $\bfv_{ik}=A_{kk} \bfw_{ik}$. 
In view of \eqref{eq:pikk}, we have
\begin{align*} 
p_{i,kk}
&= 1- v_{ik}^{\top} \big(B_{(i)}B_{(i)}^{\top}\big)^{-1}v_{ik}
= \frac{1}{1+ \bfw_{ik}^{\top} \bfM_{ik} \bfw_{ik}}\,,
\end{align*}
where we used the notation 
$$\bfM_{ik}:=A_{kk}^2 \bigg(\sum_{\ell=1; \ell \neq k}^n v_{i\ell} v_{i\ell}^{\top}\bigg)^{-1}=A_{kk}^2 \big(B_{(i,k)}B_{(i,k)}^{\top}\big)^{-1}\,$$
with $B_{(i,k)}$ denoting the matrix obtained from $B_{(i)}$ by deleting its $k$-th column $v_{ik}$.

Let $1\le i\le p-1$ and $1\le k\le n$. To simplify notation we will write $w$ instead of $\bfw_{ik}$. Multiple applications of the identity $$\frac{1}{1+x}-\frac{1}{1+y}=\frac{y-x}{(1+x)(1+y)}\,,\qquad x,y\ge 0\,,$$
give
\begin{align*}
&p_{i,kk}-\E[p_{i,kk}]=
\frac{1}{1+ w^{\top} \bfM_{ik} w}- \E\Big[ \frac{1}{1+ w^{\top} \bfM_{ik} w} \Big]\\
&= -\frac{1}{1+\tr \bfM_{ik}}+\frac{1}{1+w^{\top} \bfM_{ik} w}+
\E\Big[\frac{1}{1+\E[\tr \bfM_{ik}]}\Big]-\E\Big[\frac{1}{1+w^{\top} \bfM_{ik} w}\Big]\\
&\qquad-\frac{1}{1+\E[\tr \bfM_{ik}]}+\frac{1}{1+\tr \bfM_{ik}}\\
&= -\frac{w^{\top} \bfM_{ik} w-\tr \bfM_{ik}}{(1+w^{\top} \bfM_{ik} w)(1+\tr \bfM_{ik})}
+ \E\bigg[\frac{w^{\top} \bfM_{ik} w-\E[\tr \bfM_{ik}]}{(1+w^{\top} \bfM_{ik} w)(1+\E[\tr \bfM_{ik}])}  \bigg]\\
&\qquad -\frac{\tr \bfM_{ik}-\E[\tr \bfM_{ik}]}{(1+\tr \bfM_{ik})(1+\E[\tr \bfM_{ik}])}
\end{align*}
from which we easily deduce that
\begin{align*}
|p_{i,kk}-\E[p_{i,kk}]|&\le |w^{\top} \bfM_{ik} w-\tr \bfM_{ik}| + |\tr \bfM_{ik}-\E[\tr \bfM_{ik}]|\\
&\quad + \E[|w^{\top} \bfM_{ik} w-\tr \bfM_{ik}|] + \E[|\tr \bfM_{ik}-\E[\tr \bfM_{ik}]|].
\end{align*}
Using the triangle inequality as well as $(a+b)^2\le 2(a^2+b^2), a,b\in \R$, and Cauchy-Schwarz, we get
\begin{equation}\label{eq:sdfsds}
\E\big[(p_{i,kk}-\E[p_{i,kk}])^2\big]
\lesssim \E\big[|w^{\top} \bfM_{ik} w-\tr \bfM_{ik}|^2 \big]+\E\big[ |\tr \bfM_{ik}-\E[\tr \bfM_{ik}]|^2 \big].
\end{equation}
It remains to show that the terms on the right-hand side~ of \eqref{eq:sdfsds} are $O(1/n)$.  
Using that the fourth moment of the normal distribution is finite and $\norm{\bfA}\lesssim 1$, an application of \cite[Lemma B.26]{bai:silverstein:2010} yields for the first term that, 
\begin{align*}
\E\E \big[|w^{\top} \bfM_{ik} w-\tr \bfM_{ik}|^2\,\big| \bfM_{ik} \big]
\lesssim \E[\tr \bfM_{ik}^2]\lesssim 
\frac{n}{\lambda_{min}^2(\bfN \bfN^\top)}=O(n^{-1})\,, \qquad \nto\,,
\end{align*}
where the Cauchy interlacing theorem (see, for example, \cite[Corollary~III.1.5]{bhatia:1997}) was used to obtain the second inequality in the last line.
To bound the second term we define $\E_\ell[\cdot]:=\E\big[\cdot|\{v_{i\ell},\ldots, v_{in}\}\backslash \{v_{ik}\}\big]$ for $\ell=1,\ldots,n$ and $\E_{n+1}:=\E$. It holds
  \begin{eqnarray*}
&&    \tr\bfM_{ik} - \E\tr\bfM_{ik}=\sum_{\ell=1; \ell \neq k}^n(\E_\ell-\E_{\ell+1}) \tr\bfM_{ik}\nonumber\\
    &=& A_{kk}^2  \sum_{\ell=1; \ell \neq k}^n(\E_\ell-\E_{\ell+1})\left(\tr({B}_{(i,k)}{B}_{(i,k)}^\top)^{-1} - \tr({B}_{(i,k)}{B}_{(i,k)}^\top-{v}_{i\ell}{v}^\top_{i\ell})^{-1} \right)\,,
  \end{eqnarray*}
where we used $(\E_k-\E_{k+1}) \tr\bfM_{ik}=0$ and $(\E_\ell- \E_{\ell+1}) \tr({B}_{(i,k)}{B}_{(i,k)}^\top-{v}_{i\ell}{v}^\top_{i\ell})^{-1}=0$ for $\ell \neq k$.
Applying the Burkholder inequality for martingale differences, see e.g.\ \cite[Lemma 2.12]{bai:silverstein:2010}, we deduce
  \begin{eqnarray*}
  && \E|\tr\bfM_{ik} - \E\tr\bfM_{ik}|^2\nonumber\\
    &\lesssim& \sum_{\ell=1; \ell \neq k}^n\E\left|(\E_\ell-\E_{\ell+1})\left(\tr({B}_{(i,k)}{B}_{(i,k)}^\top)^{-1} - \tr({B}_{(i,k)}{B}_{(i,k)}^\top-{v}_{i\ell}{v}^\top_{i\ell})^{-1} \right)\right|^2\nonumber\\
    &\lesssim& \sum_{\ell=1; \ell \neq k}^n
		\E\left|(\E_\ell-\E_{\ell+1})\frac{{v}^\top_{i\ell}({B}_{(i,k)}{B}_{(i,k)}^\top-{v}_{i\ell}{v}^\top_{i\ell})^{-2}{v}_{i\ell}}{1+{v}^\top_{i\ell}({B}_{(i,k)}{B}_{(i,k)}^\top-{v}_{i\ell} {v}^\top_{i\ell})^{-1}{v}_{i\ell}}\right|^2\nonumber\\
&\lesssim&   \frac{1}{n^{2}}\sum_{\ell=1; \ell \neq k}^n \E \left|(\E_\ell+\E_{\ell+1}) \frac{{v}^\top_{i\ell}({B}_{(i,k)}{B}_{(i,k)}^\top-{v}_{i\ell}{v}^\top_{i\ell})^{-1}{v}_{i\ell}}{1+{v}^\top_{i\ell}({B}_{(i,k)}{B}_{(i,k)}^\top-{v}_{i\ell} {v}^\top_{i\ell})^{-1}{v}_{i\ell}}\right|^2\nonumber\\
    &\leq&\frac{2^2 (n-1) }{n^2}=O(n^{-1})\,, \qquad \nto\,,
  \end{eqnarray*}
where for the third line we used the Sherman-Morrison formula and for the fourth line the fact that the spectral norms $n \norm{({B}_{(i,k)}{B}_{(i,k)}^\top-{v}_{i\ell} {v}^\top_{i\ell})^{-1}}$ are uniformly bounded by a constant.
Since all our estimates were uniform in $i$ and $k$, the proof of \eqref{eq:nastypart} is complete.
\end{proof}

\section{Quadratic forms}\setcounter{equation}{0}
\begin{lemma}\cite[Lemma~1]{morales:johnstone:mckay:2021} \label{lem:johnstonelemma6}
Let $B$ be an $n\times n$ non-random symmetric matrix, $x,y\in \R^n$ random vectors of i.i.d.\ entries with mean zero, variance one, $\E|x_i|^l, \E|y_i|^l \le \nu_l$ and $\E[x_iy_i]=\rho$. Then, for any $s\ge 1$,
\begin{equation*}
\E\left| \frac{x^{\top} B y}{\|x\|_2 \|y\|_2}- \frac{\rho}{n} \tr B \right|^s
\le C_s \left[n^{-s} \big(\nu_{2s} \tr B^s + (\nu_4 \tr B^2)^{s/2} \big)+ \norm{B}^s\big(n^{-s/2} \nu_4^{s/2} +n^{-s+1} \nu_{2s}\big)  \right]\,,
\end{equation*}
where $\|\cdot \|_2$ denotes the Euclidean norm on $\R^n$ and $C_s$ is a constant depending only on $s$.
\end{lemma}

\begin{lemma}\cite[Theorem b)]{wiens:1992} \label{lem:quf}
  Let $\bfz=(Z_1,\ldots,Z_n)^\top$ be a random vector such that, for all nonnegative integers $m_1,\ldots, m_4$ with $m_1+\cdots+m_4\le 4$, $\E[Z_{i_1}^{m_1}Z_{i_2}^{m_2} \cdots Z_{i_4}^{m_4}]$ is 
(i) finite;
(ii) zero if any $m_i$ is odd; and
(iii) invariant under permutations of the indices.
Let $\beta_{2,2}=\E[Z_1^2Z_2^2], \beta_4=\E[Z_1^4]$. 	
 Then we have for any real-valued and symmetric $n\times n$ nonrandom matrices $\bfA$, $\bfB$ that
\begin{equation*}
	\E[\bfz^{\top} \bfA \bfz \bfz^{\top} \bfB \bfz]=  \beta_{2,2} [\tr \bfA \tr \bfB +2 \tr(\bfA \bfB)] + (\beta_4-3\beta_{2,2}) \tr(\bfA \circ \bfB)\,,
  \end{equation*}
  	where $\circ$ denotes the Hadamard product.
\end{lemma}


\end{document}